\documentclass[smallextended,envcountsect]{svjour3}
\usepackage{pifont}
\usepackage{epsfig,color,graphicx,subfigure,amsmath,amssymb,multirow,lscape}
\usepackage[colorlinks,linkcolor=blue,citecolor=blue]{hyperref}
\usepackage{tikz,cite}
\usetikzlibrary{arrows, automata}
\smartqed
\usepackage{graphicx}
\usepackage{amsmath,amssymb}

\usepackage{verbatim}

\journalname{}

\def\pn{\par\smallskip\noindent}

\newtheorem{lem}{Lemma}[section]

\newtheorem{rem}{Remark}[section]
\newtheorem{fact}{Fact}[section]
\newtheorem{algo}{Algorithm}[section]

\def\be{\begin{eqnarray}}
\def\ee{\end{eqnarray}}
\def\ben{\begin{eqnarray*}}
\def\een{\end{eqnarray*}}
\def\ba{\begin{array}}
\def\ea{\end{array}}
\def\bi{\begin{itemize}}
\def\ei{\end{itemize}}
\def\proof {\pn {Proof.} }

\def\endproof{\hfill $\Box$ \vskip .5cm}

\def\cS{{\mathcal S}}

\def\bR{{\mathbb R}}

 \DeclareMathOperator*{\dom}{dom}

\def\prox{{\rm Prox}}
\def\null{{\rm null}}

\def\[{\begin{equation}}
\def\]{\end{equation}}

\newcommand{\R}{\mathbb R}
\newcommand{\N}{\mathbb N}

\begin{document}

\title{Averaged proximal reflected gradient method for monotone variational inequalities}
\titlerunning{Averaged proximal reflected gradient method}
\author{Xiaokai Chang$^{1}$ \and Jialin Li$^{1}$\and Jun Yang$^{2}$}

\institute{\ding {41} Jun Yang \at xysyyangjun@163.com.\\
 This work was supported by the National Natural Science Foundation of China (12571320 and 12161053) and the Natural Science Foundation for Distinguished Young Scholars of Gansu Province (22JR5RA223).\\
1 \ School of Science, Lanzhou University of Technology, Lanzhou, Gansu, P. R. China.\\
2 \ School of Mathematics and Statistics, Xianyang Normal University,
               Xianyang, Shaanxi, P. R. China }

\date{Received: date / Revised: date / Accepted: date}

\maketitle

\begin{abstract}
Projected reflected gradient (PRG) method proposed by Malitsky is efficient for solving monotone variational inequality (MVI), while the existing upper bound of step size is not tight due to the inequality scaling in the theoretical analysis.
In this paper, we construct an averaged variant of PRG method for more general MVI and present a novel Lyapunov function to establish convergent theory. This averaged PRG method provides an improvement of the golden ratio algorithm [Y, Malitsky, Math. Program., 184, 383--410, 2020], and the involved step size is compatible with that for the classical methods, such as Popov's extragradient and forward-reflected-backward methods. Moreover, a fully adaptive strategy without linesearch is presented to adjust step sizes, which generates closed-form and potentially much larger step sizes. Numerical experiments on the Nash$-$Cournot equilibrium, HpHard, and image reconstruction problems demonstrate that the proposed algorithm significantly outperforms existing state-of-the-art methods.
\end{abstract}



\section{Introduction}
\label{sec_introduction}

In this work, we propose an algorithm for solving the monotone variational inequality (MVI) problem
\be\label{mVI}
\text{find} ~~ x^*\in \R^q ~~\text{s.t.}~~ g(x) - g(x^*) + \langle x-x^*, F(x^*)\rangle \geq 0, \quad \forall x \in \R^q ,
\ee
where we assume that
\begin{itemize}
    \item (C1) the solution set $\cS$ of \eqref{mVI} is nonempty;
    \item (C2) $g: \R^q  \to (-\infty, +\infty]$ is a proper convex lower semicontinuous (lsc) function;
    \item (C3) $F: \operatorname{dom} g \to \R^q $ is monotone and (locally) Lipschitz-continuous with constant $L_F > 0$.
\end{itemize}

The MVI is a useful way to reduce many different problems that arise in optimization, PDE, control theory, games theory to a common problem \eqref{mVI}. The important example is a convex-concave saddle point problem:
\be\label{pd-prob}
\min_{x\in \R^k}\max_{y\in \R^p} ~~g_1(x)+\Phi(x,y)-g_2(y),
\ee
where $g_1:\R^k \rightarrow(-\infty, +\infty]$ and $g_2:\R^p \rightarrow(-\infty, +\infty]$ are extended real-valued proper closed and convex functions.
$\Phi: \operatorname{dom} g_1 \times \operatorname{dom} g_2 \to \mathbb{R}$ is a smooth convex-concave function. By writing down the first-order optimality condition, it is easy to see that problem \eqref{pd-prob} is equivalent to \eqref{mVI} with $F$ and $g$ defined as
\ben
z = (x, y), \quad
F(z) = \begin{bmatrix}
\nabla_x \Phi(x, y) \\
-\nabla_y \Phi(x, y)
\end{bmatrix}, \quad
g(z) = g_1(x) + g_2(y).
\een

If $g= \iota_C$ is the indicator function of the set \(C\), that is, \(\iota_C(x) = 0\) if \(x \in C\) and otherwise \(\iota_C = \infty\), problem \eqref{mVI} reduces to
\[
\text{find } x^* \in C ~~ \text{s.t.} ~~ \langle F(x^*), y - x^* \rangle \geq 0, \ \forall y \in \R^q, \label{VI-S}
\]
where \(C\) is a closed and convex subset of \(\R^q\). Problem \eqref{VI-S} has been extensively studied in the literature, yielding weak and linear convergence results \cite{Yekini2017Strong,Yekini2020Projection,Yang2019Strong,Duong2018Weak}. For a continuously differentiable convex function \(f: \R^q \to (-\infty, +\infty)\) with its gradient denoted by \(\nabla f = F\), then problem \eqref{VI-S} is equivalent to the optimization problem
\[
\min_{x \in \R^q} f(x) + \iota_C(x).
\]

\subsection{Related methods}
It is known that, problem \eqref{mVI} is a typical example where the forward-backward (FB) method will not work.
Thus in the literature, many efficient methods for solving MVI \eqref{mVI} were modification of the FB method. One of the earliest, and arguably most popular schemes, was the extragradient (EG) method which dates back to extragradient method presented by Korpelevich \cite{Korpelevich1976} and Popov \cite{Popov1980}. Following the EG method, Tseng \cite{Tseng2000A} proposed forward-backward-forward (FBF) method, which required two values of
$F$ and only a single projection per iteration. Recently, Malitsky \cite{Malitsky2020GFBS} proposed a forward-reflected-backward (FRB) method. For a fixed stepsize $\tau > 0$, the FRB scheme can be described as
\[\label{FRB}
x_{n+1} = \prox_{\tau g}\bigl(x_n - \tau (2 F(x_n) - F(x_{n-1})\bigr),
\]
and converges weakly if the step size is chosen to satisfy $\tau < \frac{1}{2L_F}$.
Under the same assumptions as Tseng's FBF method, the FRB converges but its implementation requires only one forward evaluation per iteration instead of two.

Another method for solving \eqref{VI-S} is projected reflected gradient method \cite{PRG2015Malitsky}, which can be generalized to more general problem \eqref{mVI} with iterates
\[\label{PRG}
x_{n+1} = \prox_{\tau g}\bigl(x_n - \tau F(2x_n - x_{n-1})\bigr).
\]
This approach is termed the proximal reflected gradient (PRG for short).
PRG converges weakly when $\tau \in \bigl(0, (\sqrt{2} - 1)/L_F\bigr)$. It is easy to see that PRG \eqref{PRG} is equivalent to FRB \eqref{FRB} if operator $F$ is linear. The analysis of PRG in \cite{PRG2015Malitsky} requires $\tau \in \bigl(0, (\sqrt{2} - 1)/L_F\bigr)$, which according to the above consideration is not tight when $F$ is linear. For the special problem \eqref{VI-S}, a refined convergence analysis of the Popov's projection algorithm and PRG \cite{PRG2015Malitsky} was made in \cite{refined2025Popov} by exploiting the special properties of the projection operator, the upper bound of step size was enlarged further to $1/(2L_F)$. However, it is not clear if the upper bound of step size can be enlarged to $1/(2L_F)$ for iterate \eqref{PRG} to solve more general problem \eqref{mVI}.

Let $\phi =\frac{1+\sqrt{5}}{2}$ be the golden ratio, that is $\phi^2 = 1 + \phi$. Malitsky \cite{Malitsky2019Golden} introduced a golden ratio convex combination step
into the FB method and presented a golden ratio algorithm (GRAAL) for solving MVI problem \eqref{mVI}, having iterates
\begin{equation}\label{gr-alg}
z_{n}=\frac{\psi-1}{\psi} x_{n} + \frac{1}{\psi}z_{n-1},~~~x_{n+1} = \prox_{\tau g}(z_n - \tau F(x_{n})),
\end{equation}
where $\psi\in(1,\phi]$ and $\tau\in(0, \psi/(2L_F)]$. In \cite{Bregman2023GRA}, Bregman modification to GRAAL was presented and analysed. For the special case of $g=\iota_C$, some extensions of GRAAL and aGRAAL were considered in \cite{Beyond2023}, where the boundedness of iterates was analyzed directly via induction. Unfortunately, this induction is only for the special case of $g=\iota_C$, and does not seem to provide any route for generalization. The extensions of convex combination proposed by \cite{Malitsky2019Golden} in the primal-dual setting are also proposed in \cite{ChY2020Golden,ChY2021GoldenLinesearch,chyz-IMA} for solving saddle point problems. Moreover, a fully adaptive version of the golden ratio algorithm (aGRAAL) was presented in \cite{Malitsky2019Golden}, which does not require a linesearch to be run, and its step sizes are estimated using current information about the iterates. In practical applications of aGRAAL, stepsizes are mostly limited by its growth rate, which degrades the adaptive performance, see Section \ref{sec-experiments}.

\subsection{Motivation and Contributions}
Considering the loose step-size upper bound of the PRG \eqref{PRG} and the performance degradation of the fully adaptive step-size scheme in \cite{Malitsky2019Golden} caused by growth rate limitations, in this paper we introduce and analyze a new method for solving MVI \eqref{mVI}. For a fixed step size $\tau > 0$, the proposed scheme can be briefly described as
\[
x_{n+1} = \frac{1}{2}x_n + \frac{1}{2} \prox_{\tau g}\bigl(x_n - \tau F(2x_n - x_{n-1})\bigr).
\]
We refer to this scheme as the \textit{averaged proximal reflected gradient (aPRG) method} due to the involving of averaging step. It is worth noting that our analysis is entirely different than existing ideas for PRG \cite{PRG2015Malitsky,refined2025Popov} and GRAAL \cite{Beyond2023,Malitsky2019Golden}, and hence is of interest in its own right. Our contributions are summarized as follows.
\bi
\item[(i)] We propose a new algorithm for solving \eqref{mVI}, which is an averaged and improved version of PRG \eqref{PRG}, and introduce a rigorous analysis to explore its convergence by constructing a refined Lyapunov function.
\item[(ii)] We establish convergence theory when the operator $F$ is global Lipschitz-continuous and the step size is chosen to satisfy $\tau < \frac{1}{L_F}$. This upper bound of step size in aPRG is compatible with that for FRB \eqref{FRB} when $F$ is linear and $g\equiv0$.
\item[(iii)] An adaptive strategy without linesearch is presented for estimating step size of aPRG, where only one parameter $r\in(0,1]$ needs to be determined and the step size is computed explicitly by \eqref{est-step}. The proposed adaptive strategy differs entirely from the conventional step-size rules for gradient-based algorithms, e.g., \cite{Malitsky2019Golden,Latafat2025Adaptive}, where the minimum of two terms is adopted.
    Our experiments show that setting $r = 0.1$ is consistently a proper choice for all tested problems.
\ei

\subsection{Organization}
The organization of the remaining paper is outlined as follows. Section \ref{sec:asmp} provides basic assumptions, necessary facts, and notation. The main algorithm, an averaged PRG method with fixed step size, is introduced in Section \ref{sec: main-algorithm}. Convergence results and sublinear convergence rate results are also established in this section.
In Section \ref{sec_adaptive}, we focus on the adaptive step sizes without linesearch, and establish convergence results under this adaptive strategy.
Section \ref{sec-experiments} presents numerical results on the Nash-Cournot equilibrium, HpHard, and image reconstruction problems. Comparisons with the state-of-the-art algorithms are included as well.
Finally, Section \ref{sec_conclusion} provides some concluding remarks.

\section{Assumptions and Preliminaries}\label{sec:asmp}
Let $h$ be any extended real-valued closed proper and convex function defined on $\R^q$.
The effective domain of $h$ is denoted by $\text{dom}(h) := \{ x\in\R^q: h(x) < \infty\}$, and the subdifferential of $h$ at $x\in \R^q$ is given by $\partial h(x) := \{\xi\in\R^q: \, h(y) \geq h(x) + \langle \xi, y-x\rangle \text{~for all~}y\in\R^q\}$. Furthermore, for $\alpha >0$, the proximal operator of $\alpha h$ is given by
\begin{eqnarray*}\label{def:prox}
  \prox_{\alpha h}(x) := \arg\min_{y\in \R^m } \Big\{h(y) + {1\over 2\alpha }\|y-x\|^2\Big\}, \quad x\in \R^q,
\end{eqnarray*}
which is uniquely well defined everywhere.

An  operator $F\colon\R^q \rightarrow\R^q$ is said to be monotone if
 $$ \langle x-y, F(x)-F(y)\rangle\geq 0,~~\forall x,y\in\R^q. $$
Operator $F$ is $L_F$-Lipschitz ($L_F>0$) continuous if
\be\label{Lip}
\|F(x)-F(y)\|\leq L_F\|x-y\|,~~\forall x,y\in\R^q,
\ee
and locally continuous, i.e.,  for any bounded subsets $H\subseteq \R^q$ there exists $L_F > 0$ such that \eqref{Lip} holds for any $x,y\in H$.

The following results are elementary and can be verified or proved straightforwardly. We formalize them as following facts as they will be useful in our analysis.

\begin{fact}\label{fact_proj}
Let $h: \R^q\rightarrow (-\infty, \infty]$ be an extended real-valued closed proper and convex function. Then, for any $\tau>0$ and $x\in \R^q$,  $z = \prox_{\tau h}(x)$ if and only if
$h(y) \geq h(z)+ {1\over \tau}\langle x-z, y-z\rangle$ for all $y\in \R^q$.
\end{fact}

\begin{fact}\label{fact_ab}
Let $\{a_n\}$ and $\{b_n\}$ be two nonnegative  sequences.
\bi
\item (i) If there exists a positive integer $N$ such that $a_{n+1}\leq a_{n}-b_{n}$ for all $n\geq N$, then $\lim\limits_{n\to\infty}a_n$ exists, $\lim\limits_{n\to\infty}b_n=0$ and $\sum_{n=1}^{\infty}b_n < \infty$;
\item (ii) If there exists $\varepsilon \in (0,1)$ such that $a_{n+1}\leq \varepsilon a_{n} +b_n$ for all $n\geq1$ and $\sum\nolimits_{n=1}^{\infty}b_n < \infty$, then $\sum\nolimits_{n=1}^{\infty} a_{n}  <\infty$.
\ei
\end{fact}

\begin{fact}\label{fact_uv}
For any $u, v, a, b\in \bR$ such that $u + v > 0$, there holds $\frac{uv}{u+v}(a+b)^2\leq ua^2+v b^2$.
\end{fact}

\begin{fact}
For any $x, y, z\in \bR^q$ and $\alpha\in\bR$, we have
\be
2\langle x-y, x-z\rangle&=&  \|x-y\|^2 +  \|x-z\|^2  -\|y-z\|^2,\label{id}\\
\|\alpha x+(1-\alpha)y\|^2&=& \alpha \|x\|^2+(1-\alpha)\|y\|^2-\alpha(1-\alpha)\|x-y\|^2.\label{id2}
\ee
\end{fact}

\section{Averaged proximal reflected gradient method.}
\label{sec: main-algorithm}
This section presents the proposed proximal reflected gradient method with an averaging step and fixed step size for solving the MVI \eqref{mVI}. We then establish its convergence and convergence rate by constructing a refined Lyapunov function. The complete details of the proposed algorithm are given as follows.

\vskip5mm
\hrule\vskip2mm
\begin{algo}
[Averaged proximal reflected gradient (aPRG) method]\label{aPRG}
{~}\vskip 1pt {\rm
\begin{description}
\item[{\em Step 0.}] Let $0<\tau<\frac{1}{L_F}$.
Choose  $x_0\in \R^q$. Set $x_{-1} = x_0$ and $n=0$.
\item[{\em Step 1.}]Compute
\ben
x_{n+1} = \frac{1}{2}x_n + \frac{1}{2} \prox_{\tau g}\bigl(x_n - \tau F(2x_n - x_{n-1})\bigr).
\een
\item[{\em Step 2.}] Set $n\leftarrow n + 1$ and return to Step 1. 
  \end{description}
}
\end{algo}
\vskip1mm\hrule\vskip5mm

\begin{rem}\label{rem-cases}[Special cases of Algorithm \ref{aPRG}]
We consider three special cases where the proposed algorithm reduces to or coincides with existing known methods.
\bi
\item[(i)] If $F = 0$, Algorithm \ref{aPRG} reduces to an averaged version of proximal point algorithm, i.e.,
    \[
   x_{n+1} = \frac{1}{2}x_n + \frac{1}{2} \prox_{\tau g}\bigl(x_n\bigr).
    \]
\item[(ii)] If $g = \iota_C$, Algorithm \ref{aPRG} can be expressed as
    \[
    x_{n+1} = \frac{1}{2}x_n + \frac{1}{2} P_C\bigl(x_n - \tau F(2x_n - x_{n-1})\bigr),
    \]
    which coincides with an averaged version of \textit{projected reflected gradient method} \cite{PRG2015Malitsky}.
\item[(iii)] If $g=0$, Algorithm \ref{aPRG} becomes $x_{n+1} = x_n - \frac{\tau}{2} F(2x_n - x_{n-1})$. Under the change of variables $y_n = 2x_n - x_{n-1}$, this becomes $y_{n+1} = x_n - \tau F(y_n)$. Alternatively, Algorithm \ref{aPRG} can be expressed as the two step recursion
\[
\begin{cases}
x_{n+1} = x_n - \frac{\tau}{2} F(y_n), \\
y_{n+1} = x_{n+1} - \frac{\tau}{2} F(y_n).
\end{cases}
\]
This is exactly Popov's algorithm \cite{Popov1980} with step size $\frac{\tau}{2}$ for unconstrained variational inequality problems.
\ei
\end{rem}

Recall that the forward-backward method is not applicable to MVI \eqref{mVI}. To address this issue, the proposed aPRG method incorporates both averaging and extrapolation steps to modify the original framework, achieving guaranteed convergence and allowing step sizes compatible with existing methods. The following remark elaborates on the respective roles of the averaging and extrapolation steps.

\begin{rem}\label{rem-roles}
[Roles of averaging and extrapolation steps]
\bi
\item[(i)] If removing the 1/2-averaged step, i.e., $x_{n+1}= \prox_{\tau g}(x_{n}-\tau F(2x_n-x_{n-1}))$, Algorithm \ref{aPRG} reduces to PRG \cite{PRG2015Malitsky}, its convergence is established in \cite{PRG2015Malitsky} when $\tau\in(0, (\sqrt{2}-1)/L_F)$. Recall that the condition imposed on the step size $\tau$ is not tight; thus, we can conclude that the averaging step not only enables a larger step size but also obtains a compatible upper bound of step size.
\item[(ii)] When removing the extrapolation step, we obtain a 1/2-averaged forward-backward iteration
\be\label{averaged-only2}
x_{n+1}=\frac{1}{2}x_{n}+\frac{1}{2} \prox_{\tau g}(x_{n}-\tau F(x_n)),
\ee
which is not necessarily convergent for any $\tau>0$ from the concrete example ($g=0$) shown in Section \ref{sec:Examples}.
\ei
\end{rem}

The results presented in Remark \ref{rem-roles} shed light on the underlying mechanisms of the averaging and extrapolation steps. Specifically, the extrapolation step ensures convergence, whereas the 1/2-averaged step enables the use of large step sizes. Furthermore, leveraging this averaging framework, fully adaptive step-size strategies are developed in Section \ref{sec_adaptive}.

\subsection{Basic properties of the sequences generated by aPRG}
Recall $y_n=2x_n-x_{n-1}$ used in Remark \ref{rem-cases} again, Algorithm \ref{aPRG} can be expressed as
\[
\begin{cases}\label{alg-cc}
y_{n+1}=\prox_{\tau g}(x_n-\tau F(y_{n})),\\
x_{n+1}=\frac{1}{2} x_{n} + \frac{1}{2}y_{n+1}.
\end{cases}
\]
It is easy to observe that, scheme \eqref{alg-cc} is actually an improved version of golden ratio algorithm \cite{Malitsky2019Golden}, which broads the upper bound of convex combination parameter $\psi$ from the golden ratio value to 2 and improves step size condition from $\tau<\psi/(2L_F)$ to $\tau<1/L_F$.

By item (ii) of Fact \ref{fact_ab}, a useful property of the convex combination step $x_{n+1}=\frac{1}{2} x_{n} + \frac{1}{2}y_{n+1}$ can be observed.

\begin{lem}\label{lem-ccstep}
For any sequence $\{y_n\}$, let $\{x_n\}$ be generated by the convex combination step $x_{n+1}=\frac{1}{2} x_{n} + \frac{1}{2}y_{n+1}$ with $x_0=y_0$. If $\sum_{n=1}^{\infty}\|y_{n+1}-y_n\|^2<\infty$, then we have $\lim_{n\rightarrow\infty} \|x_n-y_n\| = 0$.
\end{lem}
\proof
Define $u_n :=x_n-y_n$, it is easy to verify from the convex combination step $x_{n+1}=\frac{1}{2} x_{n} + \frac{1}{2}y_{n+1}$, that $2 u_{n+1}- u_{n}= y_n-y_{n+1}$.
By using   \eqref{id2}, we obtain
\ben
2\|u_{n+1}\|^2-\|u_{n}\|^2+ 2\|u_{n+1}-u_n\|^2=\|y_n-y_{n+1}\|^2.
\een
This implies that $\|u_{n+1}\|^2\leq \frac{1}{2}\|u_{n}\|^2+\frac{1}{2}\|y_n-y_{n+1}\|^2$. It then follows from $\frac{1}{2}<1$,
item (ii) of Fact \ref{fact_ab} and   $\sum_{n=1}^{\infty}\|y_{n+1}-y_n\|^2<\infty$ that $\sum_{n=1}^{\infty}\|u_{n}\|^2<\infty$, which implies $\lim_{n\rightarrow\infty} \|u_n\|=\lim_{n\rightarrow\infty} \|x_n-y_n\| = 0$.
\endproof

For the MVI problem \eqref{mVI}, we define the bifunction as in \cite{Malitsky2019Golden} by
\be
 \Psi(u,v):=\langle F(u), v -u \rangle+g(v)-g(u).
\ee
It is clear that the MVI problem \eqref{mVI} is equivalent to the equilibrium problem: find $x^\star\in\R^q$ such that $\Psi(x^\star, x) \geq 0$, $\forall x\in\R^q$. Notice that for any fixed $x$, the function $\Psi(x, \cdot)$ is convex.

We define a Lyapunov function by
\be\label{defE-GRA}
E_n (x)&:=& 2\|x_{n}-x\|^2 + \zeta \| y_{n}-y_{n-1}\|^2-2\|x_{n}-y_n\|^2,
\ee
for any $x\in\R^q$ and some constant $\zeta\in(0,1)$. In sequel, we present two lemmas on the functions $\Psi(x,x_n)$ and $E_n(x)$.

\begin{lem}\label{lem1-GRA}
Let $\{(x_n, y_n)\}$ be generated by aPRG (Algorithm \eqref{aPRG}). Then for any $x\in\R^q$, we have
\[\label{anbn-relation-GRA}
2 \tau \Psi(x,y_n)+ E_{n+1}(x) \leq E_n(x)-R_n,
\]
where $R_n$ is given by
\be\label{Rn-PRG}
R_n :=(1-\zeta)2\|y_{n+1}-y_{n}\|^2,
\ee
and $\zeta\in(0,1)$ is the same as that used in \eqref{defE-GRA}.
\end{lem}
\begin{proof}
By Fact~\ref{fact_proj} and the $y$-updating in \eqref{alg-cc}, we have
\be\label{eq:4-GRA}
\langle y_{n+1}-x_{n}+\tau F(y_{n}), x - y_{n+1}\rangle\geq \tau
    (g(y_{n+1})-g(x)) \quad \forall x\in \R^q,
\ee
and
\be\label{eq:5-GRA}
\langle y_{n}-x_{n-1}+\tau F(y_{n-1}), y_{n+1} - y_{n}\rangle\geq \tau
    (g(y_{n})-g(y_{n+1})).
\ee
Note that $y_{n}-x_{n-1}=2(y_n-x_{n})$, we can rewrite \eqref{eq:5-GRA} as
\be\label{eq:6-GRA}
\langle 2(y_n-x_{n})+\tau F(y_{n-1}), y_{n+1} - y_{n}\rangle\geq \tau
    (g(y_{n})-g(y_{n+1})).
\ee
Summation of \eqref{eq:4-GRA} and \eqref{eq:6-GRA} yields
\be\label{ineq-11}
\langle y_{n+1}-x_{n}, x - y_{n+1}\rangle&+&\langle 2(y_n-x_{n}), y_{n+1} - y_{n}\rangle +\tau \langle F(y_{n})-F(y_{n-1}), y_{n}-y_{n+1}\rangle\nonumber \\
&\geq& \tau \langle F(y_{n}), y_{n} -x \rangle+\tau
    (g(y_{n})-g(x))\nonumber \\
&\geq& \tau \langle F(x), y_{n} -x \rangle+\tau
    (g(y_{n})-g(x))=\tau \Psi(x,y_n),
    \ee
where the last inequality follows the monotonicity of $F$.

Now, by applying (\ref{id}) to the first two inner products on the left-hand-side of \eqref{ineq-11}, multiplying both sides by a factor $2$ and reorganizing the terms, we derive
\be\label{ineq11}
\|y_{n+1}-x\|^2 + 2 \tau \Psi(x,y_n)
& \leq & \|x_{n}-x\|^2  + 2\tau \langle Fy_{n}- Fy_{n-1},  y_n-y_{n+1}\rangle
-2\|x_{n}-y_n\|^2 \nonumber \\
&  & + \|y_{n+1}-x_{n}\|^2   -2\|y_{n+1}-y_n\|^2.
\ee
Recall $y_{n+1}=2x_{n+1}-x_{n}$ and $x_{n+1}-x_{n} = {1\over 2} (y_{n+1} - x_{n})$ again. Then,  it is easy to deduce from (\ref{id2}) that
\be
\|y_{n+1}-x_{n}\|^2&=&4\|y_{n+1}-x_{n+1}\|^2,\label{x-to-z}\\
\|y_{n+1}-y\|^2&=&\|2(x_{n+1}-y)-(x_{n}-y)\|^2\nonumber\\
&=&2\|x_{n+1}-y\|^2- \|x_{n}-y\|^2+2\|y_{n+1}-x_{n+1}\|^2,\label{x-to-z2}
\ee
where we use $\|x_{n+1}-y_{n+1}\|^2=\|x_{n+1}-x_{n}\|^2$.
By plugging \eqref{x-to-z} and \eqref{x-to-z2} into (\ref{ineq11}), we obtain
\be\label{ineq1}
&     &  2\|x_{n+1}-x\|^2-2\|x_{n+1}-y_{n+1}\|^2 + 2 \tau \Psi(x,y_n) \nonumber \\
& \leq & 2\|x_{n}-x\|^2-2\|x_{n}-y_{n}\|^2  + 2\tau \langle Fy_{n}- Fy_{n-1},  y_n-y_{n+1}\rangle \nonumber \\
&  &   -2\|y_{n+1}-y_n\|^2.
\ee
From $\tau <1/L_F$,  i.e., $\tau=\frac{\zeta}{L_F}$ for some $\zeta\in(0,1)$, it follows that
\ben
2\tau \langle F(y_{n})-F(y_{n-1}), y_{n}-y_{n+1}\rangle&\leq& 2\tau \| F(y_{n})-F(y_{n-1})\|\|y_{n}-y_{n+1}\|\nonumber \\
&\leq& \zeta (\|y_{n}-y_{n-1}\|^2+\|y_{n+1}-y_{n}\|^2).
\een
This together with \eqref{ineq1} and the definitions of $E_n(x)$ and $R_n$ leads to \eqref{anbn-relation-GRA}. The proof is completed.
\end{proof}

In sequel, the nonnegativity of $E_{n}(x^\star)$ for any $x^\star\in\cS$ is obtained, despite the presence of the third term in $E_{n}(x^\star)$, namely $-2\|x_{n}-y_n\|^2\leq 0$. This observation is crucial for obtaining convergence and convergence rate of aPRG.

\begin{lem}\label{lem2-GRA}
Let $E_n(x)$ be defined in \eqref{defE-GRA}. For any $x^\star\in\cS$, it holds
\[\label{an-bound-GRA}
E_{n}(x^\star)\geq(2-\zeta)\|y_{n}-x^\star\|^2\geq0,
\]
where $\zeta\in(0,1)$ is the same as that used in Lemma \ref{lem1-GRA}.
\end{lem}
\proof
Recall $x_{n}-y_{n}=\frac{1}{2}(x_{n-1}-y_{n})$, it follows \eqref{id} that
\be\label{est-1-GRA}
\frac{1}{2}\|y_{n}-x_{n}\|^2+\frac{1}{2}\|{x^{\star}}-y_{n}\|^2&=& \frac{1}{2}\|x_{n}-{x^{\star}}\|^2+\langle x_{n}-y_{n}, ~ {x^{\star}}-y_{n} \rangle\nonumber\\
&=&\frac{1}{2}\|x_{n}-{x^{\star}}\|^2+\frac{1}{2}\langle x_{n-1}-y_{n}, ~ {x^{\star}}-y_{n} \rangle\nonumber\\
&\leq &\frac{1}{2}\|x_{n}-{x^{\star}}\|^2+\frac{\tau}{2}\big(\langle F (y_{n-1}), ~ {x^{\star}}-y_{n} \rangle+g({x^{\star}})-g(y_{n})\big)\nonumber\\
&=&\frac{1}{2}\|x_{n}-{x^{\star}}\|^2+\frac{\tau}{2}\langle F (y_{n-1})-F (y_n), ~ {x^{\star}}-y_{n} \rangle\nonumber\\
&&+\frac{\tau}{2}\big(\langle F(y_n), ~ {x^{\star}}-y_{n} \rangle+g({x^{\star}})-g(y_{n})\big)\nonumber\\
&\leq&\frac{1}{2}\|x_{n}-{x^{\star}}\|^2+\frac{\tau}{2}\langle F(y_{n-1})-F(y_n), ~ {x^{\star}}-y_{n} \rangle,
\ee
where the first inequality follows \eqref{eq:4-GRA} at the $n$-th iteration, and the last follows $\langle F(y_n), ~ {x^{\star}}-y_{n} \rangle+g({x^{\star}})-g(y_{n})\leq -\Psi(x^\star, y_n)\leq0$ due to $x^\star\in \cS$.

Using $\tau L_F<1$, i.e., $\tau=\frac{\zeta}{L_F}$ for some $\zeta \in (0,1)$, gives
\be\label{est-2-GRA}
\tau\langle F(y_{n-1})-F(y_n), ~ {x^{\star}}-y_{n} \rangle\leq \zeta \Big(\frac{1}{2}\|y_{n}-y_{n-1}\|^2 +  {1\over 2} \|y_{n}-x^\star\|^2\Big).
\ee
Substituting \eqref{est-2-GRA} into \eqref{est-1-GRA} and multiplying by $2$, we obtain
\ben
\|y_{n}-x_{n}\|^2+\|{x^{\star}}-y_{n}\|^2
&\leq&\|x_{n}-{x^{\star}}\|^2+\zeta \Big(\frac{1}{2}\|y_{n}-y_{n-1}\|^2 +  {1\over 2} \|y_{n}-x^\star\|^2\Big).
\een
Noticing $\zeta \in(0,1)$ and the definition of $E_{n}(x)$ gives
\ben
E_{n}(x^\star)&=&2\|x_{n}-{x^{\star}}\|^2+ \zeta \| y_{n}-y_{n-1}\|^2-2\|x_{n}-y_n\|^2\\
&\geq&(2-\zeta)\|y_{n}-x^\star\|^2>\|y_{n}-x^\star\|^2\geq0.
\een
\endproof

\subsection{Convergence results}
Based on Lemmas \ref{lem1-GRA} and \ref{lem2-GRA} for the proposed aPRG, we next establish  global iterate convergence and  ergodic sublinear convergence rate results.

\begin{theorem}\label{th:graal}
Let $\{x_n\}$ be generated by aPRG (Algorithm \ref{aPRG}), then it converges to a solution of the MVI problem \eqref{mVI}.
\end{theorem}
\proof
Let $x^\star \in \cS$ be arbitrary. From \eqref{anbn-relation-GRA}, it holds
\ben
2 \tau \Psi(x^\star,y_n)+ E_{n+1}(x^\star) \leq E_n(x^\star)-R_n.
\een
The item (i) of Fact \ref{fact_ab} implies that $\lim_{n\rightarrow\infty}E_{n}(x^{\star})$ exists, $\lim_{n\rightarrow\infty} R_n = 0$ and $\sum_{n=1}^{\infty}R_n < \infty$. By the definition of $R_n$ in \eqref{Rn-PRG}, we have
\be\label{three-lim=0}
\lim\limits_{n\rightarrow\infty}\|y_{n+1}-y_n\|=0.
\ee
and $\sum_{n=1}^{\infty}\|y_{n+1}-y_n\|^2<\infty$. This together with Lemma \ref{lem-ccstep} gives $\lim_{n\rightarrow \infty}\|x_{n}-y_n\|=0$. Referring to Lemma \ref{lem1-GRA} and the definition of $E_{n+1}(x)$, the sequence $\{x_{n}\}$ is bounded, then $\{y_{n}\}$ is bounded as well.

Thus, let $x^*\in\R^q$ be a cluster point of $\{y_n\}$. Then, there exists a
subsequence $\{y_{n_k}\}_{k\in\N}$ such that $y_{n_k}\rightarrow x^*$, $y_{n_k-1}\rightarrow x^*$ and $x_{n_k}\rightarrow x^*$ as $k\rightarrow\infty$. Now, recalling (\ref{eq:4-GRA}) and $y_{n}-x_{n-1}=2(y_{n}-x_{n})$ gives
\ben
\langle 2(y_{n_k}-x_{n_k})+\tau F(y_{n_k-1}), x- y_{n_k}\rangle\geq \tau
    (g(y_{n_k})-g(x)), ~~\forall x\in \R^q,
\een
and taking the limit-infimum of both sides as $k\rightarrow\infty$ shows that $x^*\in\cS$.

Since $x^\star\in\cS$ was chosen in Lemma \ref{lem2-GRA} to be arbitrary, we can now set $x^\star=x^*$. It then follows that $\lim_{k\rightarrow \infty} \|y_{n_k}-x^*\| = 0$, and consequently, $\lim_{k\rightarrow \infty}E_{n_k}(x^*)= 0$. Furthermore, it is clear from \eqref{anbn-relation-GRA} that the whole sequence $\{E_n(x^*)\}$ is monotonically nonincreasing. Therefore, the whole sequence $\{E_n(x^*)\}$ must converge to $0$.
As a result, we have $\lim_{n\rightarrow\infty}y_{n} = x^*$.
By noting Lemma \ref{lem-ccstep} again, we obtain $\lim_{n\rightarrow\infty}x_n = x^*$.
\endproof

We next establish the ergodic sublinear convergence rate of aPRG using the restricted merit function first proposed in \cite{Nesterov2007Dual}.
In \cite{Malitsky2019Golden}, the restricted merit function for the MVI problem \eqref{mVI} is defined as
\ben
e_r(v)&:=&\max_{u\in U}\Psi(u,v),~~\forall v\in \R^q,
\een
where $U=\dom g\cap B[\bar{y}; r]$, $\bar{y}\in \dom g$ and $r > 0$ are selected such that $U$ contains at least one solution of \eqref{mVI}. From \cite[Lemma 3]{Malitsky2019Golden}, the function $e_r$ is well defined and convex on $\R^q$. For all $y\in U$, $e_r (y)\geq0$. If $x^\star\in U$ is a solution to (\ref{mVI}), then $e_r (x^\star) = 0$. Conversely, if $e_r (\hat{y}) = 0$ for some $\hat{y}$ with $\|\hat{y} -\bar{y}\|< r$, then $\hat{y}$ is a solution of (\ref{mVI}).

Using \eqref{anbn-relation-GRA}, the convexity of $\Psi(x,\cdot)$ and Jensen's inequality gives
\ben
\Psi(x,\hat{y}_N)\leq \frac{1}{N}\sum_{n=1}^N\Psi(x,y_n) &\leq& \frac{E_1(x)- E_{N+1}(x)}{2\tau N}\\
&\leq& \frac{E_1(x)+2\|x_{N+1}-y_{N+1}\|^2}{2\tau N}, ~~\forall y\in U,
\een
where $\hat{y}_N$ is the ergodic sequence ${\hat y}_N =\frac{1}{N}\sum_{n=1}^N y_n$. Recall that $2\|x_{N+1}-y_{N+1}\|^2=\frac{1}{2}\|x_{N}-y_{N+1}\|^2$ and $\lim_{n\rightarrow \infty}\|x_{n-1}-y_n\|=0$, $2\|x_{N+1}-y_{N+1}\|^2$ is upper bounded. Since $F$ is continuous and $g$ is lsc, there exist some constant $M > 0$ that majorizes $E_1(x)+2\|x_{N+1}-y_{N+1}\|^2$ for all $y\in U$.  Hence, we obtain
\ben
e_r(\hat{y}_N)= \max_{x\in U}\Psi(x,\hat{y}_N)\leq \frac{M}{2\tau N},
 \een
which implies the $O(1/N)$ convergence rate for the ergodic sequence $\{\hat{y}_N\}$.

\section{Adaptive step sizes for aPRG}
\label{sec_adaptive}

Under the fixed step size, aPRG requires knowledge of the global Lipschitz constants of $F$, which can be challenging to obtain in practice and poor estimates of these constants can significantly deteriorate the practical performance. Moreover, even with known Lipschitz constants, the stepsizes derived from the global Lipschitz constants is usually much overconservative since they fail to utilize local geometry, resulting slow practical convergence. Thus in this section, we introduce an adaptive strategy of step sizes for aPRG.

\subsection{Fully adaptive and closed-form step sizes}
As in \cite{Latafat2025Adaptive} for Lipschitz estimates, we define
\be\label{def:Ln}
L_n=\frac{\|F(y_n)-F(y_{n-1})\|}{\|y_n-y_{n-1}\|},
\ee
for the monotone operator $F$. It is easy to observe that, $L_n$ is a local Lipschitz estimate of $F$ for a pair of points $y_{n-1}$ and $y_n$, and $L_n\leq L_F$. Noting that the denominator of $L_n$ is zero iff $F(y_n)-F(y_{n-1})= 0$, we stick to the convention $\frac{0}{0}=0$ so that $L_n$ is (well-defined, positive) real numbers.

Using this notation of $L_n$, aPRG with adaptive step sizes is summarized in the following Algorithm \ref{aPRG-adp}.

\vskip5mm
\hrule\vskip2mm
\begin{algo}
[aPRG with adaptive step sizes (aPRG-adp)]\label{aPRG-adp}
{~}\vskip 1pt {\rm
\begin{description}
\item[{\em Require:}] stepsize parameter: $r\in(0,1]$ and $\zeta\in(0,1)$ (e.g., $\zeta=1-10^{-6}$);\\
initial point $x_0\in \R^q$, $y_{0} = x_0$, and initial stepsize $\tau_0>0$.\\
\item[{\em Initialize:}] $y_{1}=\prox_{\tau_0 g}\bigl(x_0 - \tau_0 F(y_0)\bigr)$, $\tau_{-1}=\tau_0$ and $n=1$.
\item[{\em Step 1.}]With $L_n$ as in \eqref{def:Ln}, estimate step size $\tau_n$ by
\be\label{est-step}
\tau_{n}=\frac{3(4-r)\tau_{n-1}}{8+3\tau_{n-1}\tau_{n-2} L_{n}^2/(r\zeta)}.
\ee
\item[{\em Step 2.}] Compute
\ben
y_{n+1}=\prox_{\tau_n g}\bigl(x_n - \tau_n F(y_n)\bigr)~~\mbox{and}~~x_{n+1} = \frac{1}{2}x_n + \frac{1}{2} y_{n+1}.
\een
\item[{\em Step 3.}] Set $n\leftarrow n + 1$ and return to Step 1. 
  \end{description}
}
\end{algo}
\vskip1mm\hrule\vskip5mm

Remarkably, aPRG-adp adopts a closed-form step size, which differs entirely from the existing step-size rules for gradient-based algorithms, e.g., \cite{Malitsky2019Golden,Latafat2025Adaptive}, where the minimum of two terms is adopted. While the expression \eqref{est-step} of $\tau_n$ looks complicated, it merely consists of real-valued arithmetic operations. Define $\rho:=\frac{3(4-r)}{8}$. It is easy to observe from the rule \eqref{est-step} that, if $L_n=0$ it holds
\ben
\tau_{n}=\frac{3(4-r)}{8}\tau_{n-1}=\rho\tau_{n-1}.
\een
This gives $\delta_n:=\frac{\tau_{n}}{\tau_{n-1}}<\rho$ from the updating \eqref{est-step} for the case of $L_n>0$, hence the growth rate of step sizes is upper bounded by $\rho$. The introduction of constant $\zeta<1$ enforced in aPRG-adp is for establishing sufficient descent of the energy function $\widetilde{E}_n(x)$, see the definition of $\widetilde{R}_n$ in Lemma \ref{lem-ls}.

\begin{rem}
We now provide several remarks on aPRG-adp.
\bi
\item[(i)] \textbf{(Parameter $r$)}. It should be emphasized that aPRG-adp contains only one unique parameter $r$, whose value directly governs the step size generation mechanism.
    Recall $r\in(0,1]$, the corresponding upper bound of $\rho$ lies in the interval $[9/8, 3/2]$. It can be readily seen that a smaller value of $r$ yields a larger growth factor $\rho$, which can improve numerical performance, however taking $r$ too small may prevent adopting large step sizes in \eqref{est-step}.
\item[(ii)] \textbf{(Specific rule)}. The rule in \eqref{est-step} is derived from condition \eqref{results-step}, i.e., $\tau_n^2L_n^2= \zeta r  (4-r-\frac{8}{3}\delta_n)\delta_n \delta_{n-1}$, which is constructed to derive the conclusion of Lemma \ref{lem-ls}. Let $\widetilde{\rho}\in(1,\rho)$ and $\delta_n\leq \widetilde{\rho}$, the condition  \eqref{results-step} can be replaced by $\tau_n^2L_n^2\leq \zeta r  (4-r-\frac{8}{3}\widetilde{\rho})\delta_n \delta_{n-1}$ (tighter than condition \eqref{results-step}), then a rule
    \ben
    \tau_n=\left\{\widetilde{\rho}\tau_{n-1}, ~~\frac{\zeta r  (4-r-\frac{8}{3}\widetilde{\rho})}{\tau_{n-2}L_n^2}\right\}
    \een
    can be used, which adopted the minimum of two terms and is similar with that used in \cite[Algorithm 1]{Malitsky2019Golden}. Numerical experiments demonstrate that, larger $\widetilde{\rho}$ may result in better performance of aPRG-adp. Accordingly, we may adopt a relatively large value of $\widetilde{\rho}$, such as $\widetilde{\rho}=\frac{11}{8}$, to improve numerical performance. Meanwhile, we set $r=\frac{1}{6}$ so that $r(4-r-\frac{8}{3}\widetilde{\rho})$ with $\widetilde{\rho}=\frac{11}{8}$ attains its maximum value $\frac{1}{36}$. Under this setting, step size can be given by specific rule
\be\label{step-special}
\tau_{n}=\min\left\{\frac{11}{8}\tau_{n-1}, ~~\frac{\zeta}{36\tau_{n-2}L_n^2}\right\}.
\ee
We observe that the update rule \eqref{est-step} outperforms scheme \eqref{step-special} and consistently achieves favorable numerical performance across all test problems; detailed numerical comparisons are provided in Section \ref{sec-experiments}.
\ei
\end{rem}

\begin{lem}\label{lem-ls}
Let $\{(x_n, y_n)\}$ be generated by aPRG-adp (Algorithm \ref{aPRG-adp}). For any $x\in\R^q$, we have
\[\label{anbn-relation-GRA-adp}
2 \tau_n \Psi(x,y_n)+ \widetilde{E}_{n+1}(x) \leq \widetilde{E}_n(x)-\widetilde{R}_n,
\]
where
\be\label{Rn-PRG-adp}
\left\{\ba{rcl}\widetilde{E}_n(x)&:=& 2\|x_{n+1}-x\|^2+ r\delta_{n-1}\|y_n-y_{n-1}\|^2, \\
\widetilde{R}_n &:=&(1-\zeta)r\delta_{n-1}\|y_{n}-y_{n-1}\|^2.
\ea\right.\ee
\end{lem}
\begin{proof}
By Fact~\ref{fact_proj} and $y_{n+1}=\prox_{\tau_n g}\bigl(x_n - \tau_n F(y_n)\bigr)$, we have
\be\label{eq:4-GRA-adp}
\langle y_{n+1}-x_{n}+\tau_n F(y_{n}), x - y_{n+1}\rangle\geq \tau_n
    (g(y_{n+1})-g(x)) \quad \forall x\in \R^q,
\ee
and similarly
\be\label{eq:5-GRA-adp}
\langle y_{n}-x_{n-1}+\tau_{n-1} F(y_{n-1}), y_{n+1} - y_{n}\rangle\geq \tau_{n-1}
    (g(y_{n})-g(y_{n+1})).
\ee
Note that $y_{n}-x_{n-1}=2(y_n-x_{n})$ and $\delta_n=\frac{\tau_{n}}{\tau_{n-1}}$, we can rewrite \eqref{eq:5-GRA-adp} as
\be\label{eq:6-GRA-adp}
\langle 2\delta_n(y_n-x_{n})+\tau_{n} F(y_{n-1}), y_{n+1} - y_{n}\rangle\geq \tau_{n}
    (g(y_{n})-g(y_{n+1})).
\ee
Summation of \eqref{eq:4-GRA-adp} and \eqref{eq:6-GRA-adp} yields
\be\label{ineq-11-adp}
\langle y_{n+1}-x_{n}, x - y_{n+1}\rangle&+&\langle 2\delta_n(y_n-x_{n}), y_{n+1} - y_{n}\rangle +\tau_n \langle F(y_{n})-F(y_{n-1}), y_{n}-y_{n+1}\rangle\nonumber \\
&\geq& \tau_n \langle F(y_{n}), y_{n} -x \rangle+\tau_n
    (g(y_{n})-g(x))\nonumber \\
&\geq& \tau_n \langle F(x), y_{n} -x \rangle+\tau_n
    (g(y_{n})-g(x))=\tau_n \Psi(x,y_n),
    \ee
where the last inequality follows the monotonicity of $F$.

Now, by applying (\ref{id}) to the first two inner products on the left-hand-side of \eqref{ineq-11-adp}, multiplying both sides by a factor $2$ and reorganizing the terms, we derive
\be\label{ineq11-adp}
\|y_{n+1}-x\|^2 + 2 \tau_n \Psi(x,y_n)
& \leq & \|x_{n}-x\|^2  + 2\tau_n \langle Fy_{n}- Fy_{n-1},  y_n-y_{n+1}\rangle
-2\delta_n\|x_{n}-y_n\|^2 \nonumber \\
&  & + (2\delta_n-1)\|y_{n+1}-x_{n}\|^2   -2\delta_n\|y_{n+1}-y_n\|^2.
\ee
By plugging \eqref{x-to-z} and \eqref{x-to-z2} into (\ref{ineq11-adp}), we obtain
\be\label{ineq1-adp}
2\|x_{n+1}-x\|^2 + 2 \tau_n \Psi(x,y_n)
& \leq & 2\|x_{n}-x\|^2  + 2\tau_n \langle Fy_{n}- Fy_{n-1},  y_n-y_{n+1}\rangle
-2\delta_n\|x_{n}-y_n\|^2 \nonumber \\
&  & + (2\delta_n-\frac{3}{2})\|y_{n+1}-x_{n}\|^2   -2\delta_n\|y_{n+1}-y_n\|^2.
\ee
Using Fact \ref{fact_uv} with $u=2\delta_n>0$ and $v=\frac{3}{2}-2\delta_n$ so that $u+v = \frac{3}{2} >0$,
$a=\|x_{n}-y_n\|$, $b=\|x_{n}-y_{n+1}\|$ and noting   $\|y_{n+1}-y_{n}\|\leq  a+b$, we can easily obtain
\be\label{ineq-relax}
2\delta_n\|x_{n}-y_n\|^2+(\frac{3}{2}-2\delta_n) \|y_{n+1}-x_{n}\|^2
\geq2\delta_n(1-\frac{4\delta_n}{3}) \|y_{n+1}-y_n\|^2.
\ee
Recall the definition $\widetilde{E}_n(x)$, from \eqref{ineq1-adp} and \eqref{ineq-relax} we can deduce
\be\label{bound-E}
\widetilde{E}_{n+1}(x) + 2 \tau_n \Psi(x,y_n)
& \leq & \widetilde{E}_n(x)  + 2\tau_n \langle Fy_{n}- Fy_{n-1},  y_n-y_{n+1}\rangle
-r\delta_{n-1}\|y_{n}-y_{n-1}\|^2 \nonumber \\
&  & - (4-r-\frac{8}{3}\delta_n)\delta_n\|y_{n+1}-y_n\|^2.
\ee
For the case of $L_n=0$, i.e., $Fy_{n}- Fy_{n-1}=0$, it follows from  \eqref{est-step} that $4-r-\frac{8}{3}\delta_n=0$ and then the result \eqref{anbn-relation-GRA-adp} holds due to $r\delta_{n-1}\|y_{n}-y_{n-1}\|^2> \widetilde{R}_n$ defined in \eqref{Rn-PRG-adp}.

Now, we check the case of $L_n>0$. From $\tau_n< \rho \tau_{n-1}$ and $\rho=\frac{3(4-r)}{8}$, we have $\delta_n< \frac{3(4-r)}{8}$ and then $4-r-\frac{8}{3}\delta_n>0$. By  \eqref{est-step}, it holds
\ben
\tau_n(\tau_{n-1}^2L_n^2+\frac{8}{3}\zeta r\delta_{n-1}) =  r  (4-r)\zeta\tau_{n-1} \delta_{n-1}.
\een
That is,
\be\label{results-step}
\tau_n^2L_n^2= \zeta r  (4-r-\frac{8}{3}\delta_n)\delta_n \delta_{n-1}.
\ee
We bound the term $2\tau_n \langle Fy_{n}- Fy_{n-1},  y_n-y_{n+1}\rangle$ by Young's inequality with parameter $\varepsilon_n>0$ as
\ben
2\tau_n \langle Fy_{n}- Fy_{n-1},  y_n-y_{n+1}\rangle\leq \varepsilon_n\tau_n^2 \| Fy_{n}- Fy_{n-1}\|^2+\frac{1}{\varepsilon_n}\|y_n-y_{n+1}\|^2.
\een
Selecting $\varepsilon_n=\frac{1}{(4-r-\frac{8}{3}\delta_n)\delta_n}>0$ and using \eqref{results-step}, we have
\ben
2\tau_n \langle Fy_{n}- Fy_{n-1},  y_n-y_{n+1}\rangle\leq \zeta r\delta_{n-1} \|y_{n}- y_{n-1}\|^2+(4-r-\frac{8}{3}\delta_n)\delta_n\|y_n-y_{n+1}\|^2,
\een
which together with \eqref{bound-E} gives
\ben
\widetilde{E}_{n+1}(x) + 2 \tau_n \Psi(x,y_n)
& \leq & \widetilde{E}_n(x)  -(1-\zeta)r\delta_{n-1}\|y_{n}-y_{n-1}\|^2.
\een
Combining the two cases $L_n=0$ and $L_n>0$, we complete the proof.
\end{proof}

We next establish some important properties of step size sequence $\{\tau_n\}$ generated by aPRG-adp. Following \cite{Beyond2023}, we call the first option when $\tau_{n-1}\tau_{n-2}L_n^2\leq \frac{r\zeta^2}{3}$, from which we deduce $\tau_n\geq \rho_1\tau_{n-1}$ with $\rho_1:=\frac{3(4-r)}{(8+\zeta)}>1$; we call the second option otherwise, where $\tau_n=s_n:=\frac{3(4-r) \tau_{n-1}}{8+3\tau_{n-1}\tau_{n-2} L_{n}^2/(r\zeta)}$. We now make following two claims:
\bi
\item \textbf{Claim 1.} For any $n \geq 1$, $\tau_n$ satisfies $\tau_n^2L_n^2= \zeta r  (4-r-\frac{8}{3}\delta_n)\delta_n \delta_{n-1}$. When $\tau_{n-1}\tau_{n-2}L_n^2> \frac{r\zeta^2}{3}$ (the second option), we have $\delta_n\leq\rho_1$ and $\tau_{n-2} + \tau_n \geq 2\sqrt{\tau_{n-2}\tau_n} =2\sqrt{\frac{\tau_{n-2}}{\tau_{n-1}}\tau_{n-1}\tau_n}=2 \sqrt{\frac{1}{\delta_{n-1}}\frac{\tau_{n-1}}{\tau_n}}\tau_n=\frac{2\sqrt{\zeta r  (4-r-\frac{8}{3}\delta_n)}}{L_n} \geq \frac{c_1}{L_F}$ for $c_1=2\sqrt{\zeta^2 r  \frac{4-r}{8+\zeta}}>0$.
\item \textbf{Claim 2.} Let $\underline{\tau}:=\frac{c_1}{L_F}$ and choose $\tau_{-1}=\tau_0 \geq \underline{\tau}/2$.
\ei
In addition, it establishes some important properties on $\{\tau_n\}$ and
$\{ \delta_n\}$, which are essential for establishing the convergence
results.

\begin{lem}\label{lem_bound}
The following claims hold.
\bi
\item[(i)] For any integer $N > 0$, we have $\sum_{i=0}^{N} \tau_i \geq C(N+ 1)$ with $C=\underline{\tau}/2>0$;
\item[(ii)] For any $\varpi \in (0, 1)$, there exists an infinite subsequence $\{ n_j : j \geq 1\} \subseteq \{1, 2, \cdots\}$ such that $\delta_{n_j} \geq \varpi$ and $\tau_{n_j} \geq C$.
\ei
\end{lem}
\begin{proof}
(i) Following \cite[Section 4]{Beyond2023}, for $k \geq -1$ let us call by a \textit{tail} a maximal subsequence of consecutive elements $\tau_{k+1},\dots,\tau_{k+m} (m\geq 2)$ such that it starts from $\tau_{k+1} \geq \underline{\tau}/2$ and the rest of the elements are all smaller:
\[
\tau_j < \underline{\tau}/2 \quad \text{for all } j = k+2,\dots,k+m \quad \text{and} \quad \tau_{k+m+1} \geq \underline{\tau}/2,
\]
where $\underline{\tau}$ is defined in \textbf{Claim 2}. Notice that for every such a tail, $\tau_{k+2} < \tau_{k+1}$, which means that the second option for $\tau_{k+2}$ is active. By \textbf{Claim 1}, this implies that $\tau_k \geq \underline{\tau}/2$ due to $\tau_{k+2}< \underline{\tau}/2$. As a result, we can partition step size sequence $\{\tau_i\}_{i=0}^N$ into a non-overlapping sequence of the tail pairs (possibly empty), and elements larger than $\underline{\tau}/2$. It is sufficient to show the bound for the tail pairs in our partition. For a tail sequence $\tau_{k+1},\dots,\tau_{k+m}$, we will show that
$$
\sum_{i=k}^{s+m} \tau_i \geq (m+1) \underline{\tau}/2.
$$

Since the second option is active for $\tau_{k+2}$, as we have already mentioned in \textbf{Claim 1}, $\tau_k \geq \underline{\tau}/2$ and $\tau_{k+2} \geq \frac{c_1^2}{4\tau_k L_F^2}$. Thus, we also incorporate $\tau_k$ into this tail. For elements $\tau_{k+4},\dots,\tau_{k+m}$ only the first option can occur, since otherwise there will be a contradiction with \textbf{Claim 1}. For $\tau_{k+3}$, we consider the first and second options for $\tau_{k+3}$, i.e.,
\[
\tau_{k+3} \geq \rho_1 \tau_{k+2} \geq \frac{\rho_1 c_1^2}{4\tau_k L_F^2} \quad \text{or} \quad \tau_{k+3} \geq \frac{c_1^2}{4\tau_{k+1} L_F^2}.
\]
Following \cite[Section 4]{Beyond2023} again, for both two cases we can obtain
$\sum_{i=k}^{k+m} \tau_i \geq(m + 1)\underline{\tau}/2.$ Hence for each subsequence of length $N+1$, the sum of its elements satisfies $\sum_{i=0}^{N} \tau_i \geq C(N+ 1)$ with $C=\underline{\tau}/2$.

(ii) For any $\varpi\in(0,1)$, define a subsequence $S:=\{n_i : i \geq 1\} = \{n : \delta_n \geq \varpi\}$. If $\{n_i : i \geq 1\}$ is not infinite, there exists a sufficiently large $M_1\geq 1$ such that $\delta_n < \varpi < 1$ for all $n\geq M_1$. This implies for any $N>M_1$ that,
\ben
\sum_{i=0}^{N} \tau_i &=&\sum_{i=0}^{M_1-1} \tau_i+\sum_{i=M_1}^{N} \tau_i\\
&\leq&\sum_{i=0}^{M_1-1} \tau_i+\tau_{M_1}\sum_{i=M_1}^{N} \varpi^{i-M_1}\\
&\leq&\sum_{i=0}^{M_1-1} \tau_i+\frac{\tau_{M_1}}{1-\varpi}<+\infty,
\een
which contradicts with (i).

In sequel, we prove that, there exist an infinite subsequence $\{n_j : j \geq 1\}$ of $S$ such that $\tau_{n_j} \ge C$. Assume, for contradiction, that subsequence $\{n_j : j \geq 1\}$ is not infinite. Then there exists $M_2 \ge 1$ such that $\tau_{n_i} < C$ for all ${n_i} \geq M_2$ and $n_i\in S$. Noticing $\delta_{n_i}\geq \varpi$ for all $n_i\in S$, all step sizes after increment or no update ($\delta_{n_i}\geq1$) are less than $C$ when ${n_i} \geq M_2$, this implies that $\tau_{n} < C$ for all $n\geq M_2$. For $N > M_2$, we have
\[
\sum_{i=1}^N \tau_i = \sum_{i=1}^{M_2-1} \tau_i + \sum_{i=M_2}^N \tau_i < \sum_{i=1}^{M_2-1} \tau_i + C(N-M_2+1).
\]
Dividing both sides by $N$:
\[
\frac{1}{N}\sum_{i=1}^N \tau_i < \frac{1}{N}\sum_{i=1}^{M_2} \tau_i + C\left(1-\frac{M_2-1}{N}\right).
\]
As $N \to \infty$, the right-hand side tends to $C$, which contradicts $\sum_{i=0}^{N} \tau_i \geq C(N+ 1)$. This completes the proof.
\end{proof}

\begin{theorem}\label{th:graal}
Let $\{x_n\}$ be generated by aPRG-adp (Algorithm \ref{aPRG-adp}), then it converges to a solution of the MVI problem \eqref{mVI}.
\end{theorem}
\proof
Let $ \varpi \in (0, 1)$ and $\underline{\tau} > 0$ be defined in \textbf{Claim 2}. By (ii) of Lemma \ref{lem_bound}, there exists an infinite sequence $\{ n_k : k \geq 1\}$ such that $\tau_{n_k} \geq \underline{\tau}/2$ and $\delta_{n_k} \geq \varpi$.
Let $x^\star \in \cS$ be arbitrary. From \eqref{anbn-relation-GRA-adp} and $\Psi(x^\star,y_n)\geq 0$, it holds
\ben
\widetilde{E}_{n+1}(x^\star) \leq \widetilde{E}_n(x^\star)-\widetilde{R}_n.
\een
The item (i) of Fact \ref{fact_ab} implies that $\lim_{n\rightarrow\infty}\widetilde{E}_{n}(x^{\star})$ exists, $\lim_{n\rightarrow\infty} \widetilde{R}_n = 0$ and $\sum_{n=1}^{\infty}\widetilde{R}_n < \infty$. By the definition of $\widetilde{R}_n$ in \eqref{Rn-PRG-adp}, we have
\be\label{three-lim=0}
\lim\limits_{n\rightarrow\infty}\|y_{{n_k}+1}-y_{n_k}\|=0.
\ee
and $\sum_{n=1}^{\infty}\|y_{{n_k}+1}-y_{n_k}\|^2<\infty$. This together with Lemma \ref{lem-ccstep} gives $\lim_{n\rightarrow \infty}\|x_{n_k}-y_{n_k}\|=0$. Referring to Lemma \ref{lem-ls} and the definition of $\widetilde{E}_{n+1}(x)$, the sequence $\{x_{n}\}$ is bounded, then $\{y_{n}\}$ is bounded due to $y_n=2x_n-x_{n-1}$.

Let $x^*\in\R^q$ be a cluster point of $\{y_n\}_{n\in\N}$. Then, there exists a
subsequence of $\{ n_k : k \geq 1\}$, still denoted as $\{ n_k : k \geq 1\}$, such that $y_{n_k}\rightarrow x^*$, $y_{n_k-1}\rightarrow x^*$ and $x_{n_k}\rightarrow x^*$ as $k\rightarrow\infty$. Now, recalling (\ref{eq:4-GRA}) and $y_{n}-x_{n-1}=2(y_{n}-x_{n})$ gives
\ben
\langle 2(y_{n_k}-x_{n_k})+\tau_{n_k} F(y_{n_k-1}), x- y_{n_k}\rangle\geq \tau_{n_k}
    (g(y_{n_k})-g(x)), ~~\forall x\in \R^q,
\een
and taking the limit-infimum of both sides as $k\rightarrow\infty$ shows that $x^*\in\cS$.

Since $x=x^\star\in\cS$ was chosen in Lemma \ref{lem-ls} to be arbitrary, we can now set $x^\star=x^*$. It then follows that $\lim_{k\rightarrow \infty} \|y_{n_k}-x^*\| = 0$, and consequently, $\lim_{k\rightarrow \infty}\widetilde{E}_{n_k}(x^*)= 0$. Furthermore, it is clear from \eqref{anbn-relation-GRA-adp} that the whole sequence $\{\widetilde{E}_n(x^*)\}$ is monotonically nonincreasing. Therefore, the whole sequence $\{\widetilde{E}_n(x^*)\}$ must converge to $0$.
As a result, we have $\lim_{n\rightarrow\infty}y_{n} = x^*$.
By noting Lemma \ref{lem-ccstep} again, we obtain $\lim_{n\rightarrow\infty}x_n = x^*$.

\endproof
\section{Numerical results}
\label{sec-experiments}
In this section, numerical experiments are presented to assess the performance of aPRG (Algorithm \ref{aPRG}) and its adaptive version, i.e., aPRG-adp (Algorithm \ref{aPRG-adp}). Moreover, we test aPRG-adp using the specific rule \eqref{step-special} and denote by aPRG-adp-sr.
The first test is the benchmark example to elaborate the respective roles of the averaging and extrapolation steps. The other tests, including Nash-Cournot equilibrium, HpHard and Tomography reconstruction problems, are conducted to illustrate the efficiency of the proposed method. We also compare the proposed aPRG and aPRG-adp with the following methods:
\bi
\item PRG (proximal reflected gradient method \cite{PRG2015Malitsky}) with step size $\lambda=(\sqrt{2}-1)/L_F$;
\item TFBF-L (Tseng's forward backward forward with linesearch);
\item FRB-L (forward reflected backward method with linesearch \cite{Malitsky2020GFBS});
\item GRAAL ($\psi = 1.6$ and step size $\lambda=\psi/(2L_F)$) and aGRAAL (adaptive step sizes) \cite{Malitsky2019Golden};
\ei

All the experiments were carried out on a 64-bit Windows system with an Intel(R) Core(TM) i5-4590 processor (3.30 GHz) and 8 GB RAM, and all the results are reproducible by specifying the \texttt{seed} of the random number generator in the code accessible at \url{https://github.com/xkchang-opt/aPRG-adp}.

\subsection{Example for exploring the 1/2-averaged step}
\label{sec:Examples}
We test a specific MVI problem \eqref{mVI} in $\R^2$ with $g=0$ and $F(x)=Mx$ with
\ben
M=\left[\ba{cc}0,~&~-1\\ 1,~&~0\ea\right].
\een
Applying the 1/2-averaged FB iteration, we have
\ben
x_{n+1}=\frac{1}{2}x_{n}+\frac{1}{2}(x_n-\tau F(x_n))=x_n-\frac{\tau}{2} Mx_n=(I-\frac{\tau}{2}M )x_n,
\een
with step size $\tau>0$. Since the two eigenvalues of $I-\frac{\tau}{2}M $ are given by $\lambda(I-\frac{\tau}{2}M )= 1\pm \frac{\tau i}{2}$, thus it holds $|\lambda(I-\frac{\tau}{2}M)|=\sqrt{1+\tau^2/2}>1$ for any $\tau>0$. This implies that the 1/2-averaged forward-backward iteration \eqref{averaged-only2} is not necessarily convergent. In other words, the 1/2-averaged step alone cannot guarantee the convergence of the forward-backward algorithm for solving the MVI problem \eqref{mVI}.

\begin{figure}[htp]
\centering\subfigure[Results obtained from averaged FB and aPRG.]{
\includegraphics[width=0.45\textwidth]{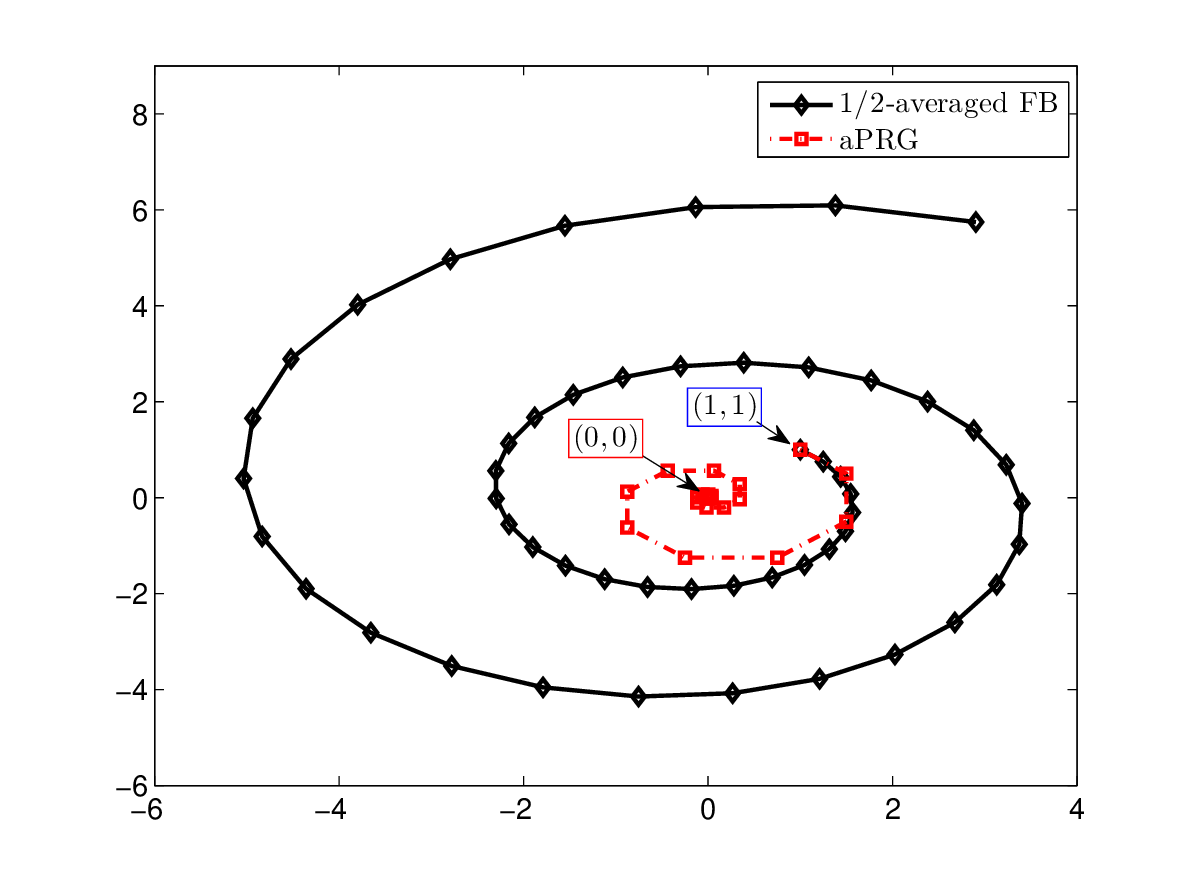}}
\subfigure[$\|x_n-(0,0)\|$]{
\includegraphics[width=0.45\textwidth]{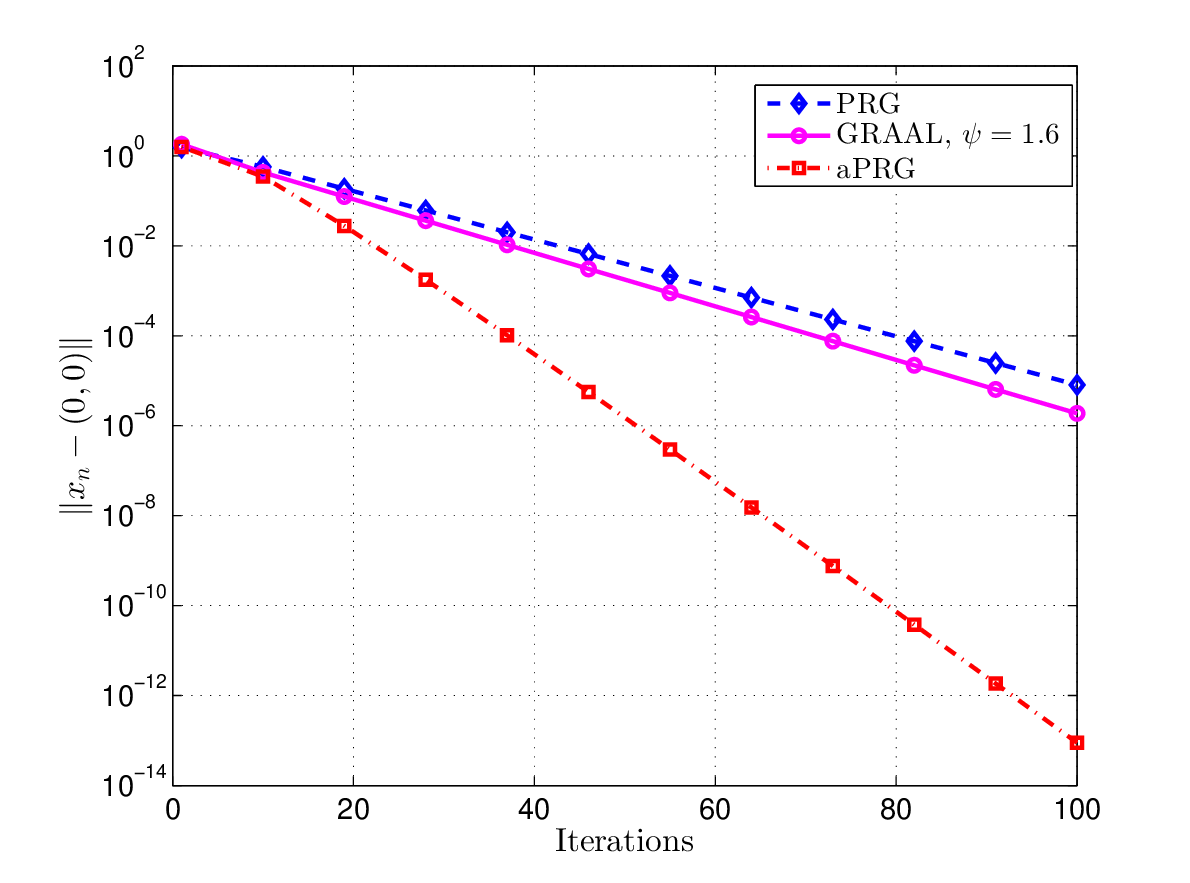}}
\caption{Results for the special MVI problem with $g=0$ and $F(x)=Mx$. }
\label{Fig:M}
\end{figure}

Since the specific MVI problem \eqref{mVI} with $g=0$ and $F(x)=Mx$ has the unique solution point $x^\star=(0,0)$, we tested PRG, GARRL with $\psi=1.6$ and aPRG, and compared the results of residual $\|x_n-(0,0)\|$. The initial point is selected as $(x_0,y_0)=(1,1)$. From Figure \ref{Fig:M} (a), the 1/2-averaged FB is divergent for this problem, and aPRG converges to the solution $x^\star=(0,0)$. The plot of residual $\|x_n-(0,0)\|$ shown in Figure \ref{Fig:M} (b) illustrates aPRG has better performance than GRAAL and PRG, due to larger step size.

\subsection{Nash--Cournot equilibrium}

In this section, we test the Nash-Cournot oligopolistic equilibrium model as in \cite[Section 5.1]{Malitsky2019Golden}. There are $m$ firms, each of them supplies a homogeneous product in a non-cooperative fashion. Let $q_i \geq 0$ denote the $i$-th firm's supply at cost $f_i(q_i)$ and $Q = \sum_{i=1}^m q_i$ be the total supply in the market. Let $p(Q)$ denote the inverse demand curve. A variational inequality that corresponds to the equilibrium is
\ben
\text{find } q^* = (q_1^*, \dots, q_m^*) \in \mathbb{R}_+^m~~ \text{ s.t. }~~ \langle F(q^*), q - q^* \rangle \geq 0, \quad \forall q \in \mathbb{R}_+^m,
\een
where $F(q^*) = (F_1(q^*), \dots, F_n(q^*))$ and
\ben
F_i(q^*) = f_i'(q_i^*) - p\Bigl(\sum_{j=1}^m q_j^*\Bigr) - q_i^* p'\Bigl(\sum_{j=1}^m q_j^*\Bigr).
\een

Following \cite{Malitsky2019Golden}, we assume that the inverse demand function $p$ and the cost function $f_i$ take the form:
\[
p(Q) = 5000^{1/\gamma} Q^{-1/\gamma} \quad \text{and} \quad f_i(q_i) = c_i q_i + \frac{\beta_i}{\beta_i+1} L_i^{\frac{1}{\beta_i}} q_i^{\frac{\beta_i+1}{\beta_i}},
\]
where we set $\gamma = 1.5$, $\beta_i \sim \mathcal{U}(0.3, 4)$, $c_i \sim \mathcal{U}(1, 100)$ and $L_i \sim \mathcal{U}(0.5, 5)$. In our tests, we set $m =500$ and $m= 1000$ and generate data randomly.

For comparison we use the residual
\be\label{def:r_n}
r_n:= \|y_{n+1}- P_{\mathbb{R}_+^m}(x_n - F(y_n))\|
\ee
to terminate algorithms when $r_n\leq \epsilon$, which we compute in every iteration. The starting point is $x_0 = (1,\dots,1)$. Furthermore, all the tested algorithms were terminated as well if a maximum number of iterations, named $n_{\max}$, was reached.
In this set of experiments, we set $\epsilon=10^{-8}$ and $n_{\max}=5\times 10^5$.

We first explored the performance of aPRG-adp with different $r$, and ran a
grid search for $r\in [0.005, 0.95]$. From the results of iterations, we observed that the performance of aPRG-adp is generally sensitive to the choice of $r$, but aPRG-adp with $r= 0.1$ performed well enough on the presented examples. The results of iterations from aPRG-adp with $r\in\{0.005, 0.01, 0.05, 0.1, 0.4, 0.7\}$, for solving the problem with $m=500$, are given in Figure \ref{Fig:test-r}. Thus in this section, we set $r=0.1$ for all the tested problems.

\begin{figure}[htbp]
\centering
\subfigure[$m=500$, \texttt{seed}=1]{
\includegraphics[width=0.4\textwidth]{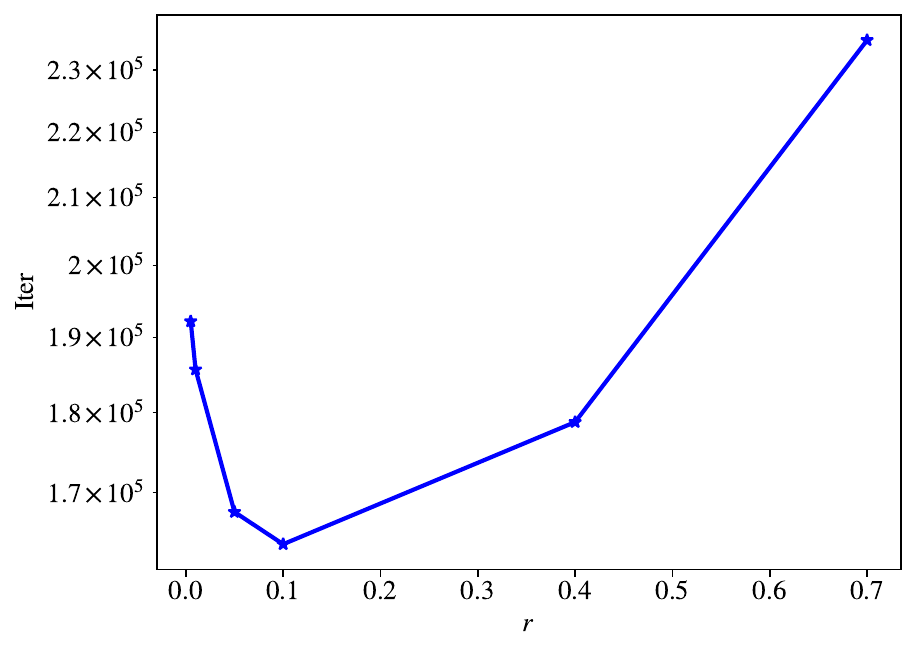}}
\subfigure[$m=500$, \texttt{seed}=2]{
\includegraphics[width=0.4\textwidth]{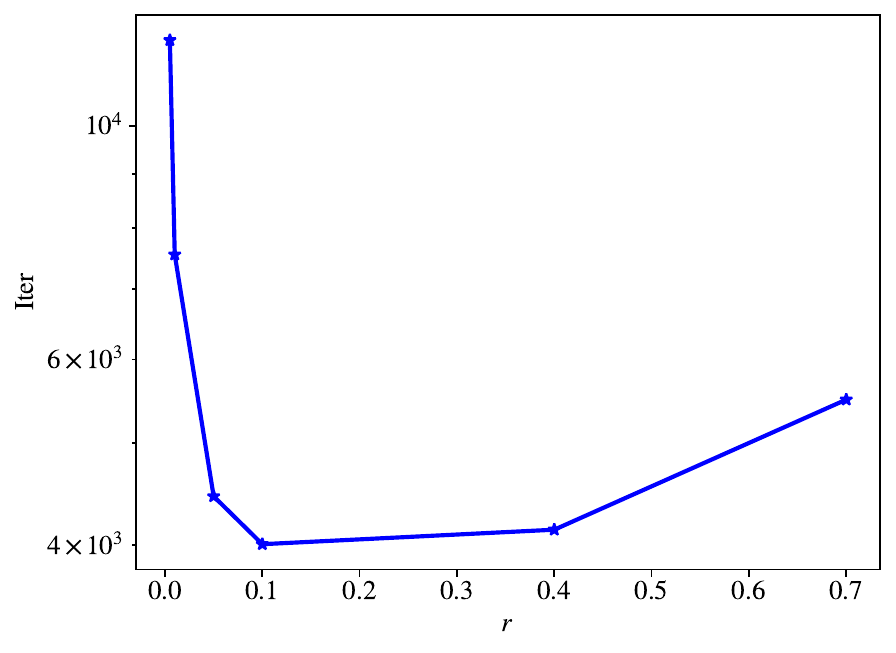}}
\caption{Results of iterations from aPRG-adp with different $r$. }
\label{Fig:test-r}
\end{figure}

We compare aPRG-adp with aPRG-adp-sr, aGRAAL, FRB and Tseng's FBF method with linesearch. Table \ref{table-Nash} reports the number of iterations (\texttt{Iter}), total CPU time (\texttt{Time}, in seconds), and the number of extra linesearch trial steps (\#LS) required by FRB-L. The results of residual $r_n$ with respect to \texttt{Time} are plotted in Figure \ref{Fig:Nash} for two cases. Since Tseng's FBF can not terminate within the maximum number of iterations, and aPRG-adp-sr delivers comparable yet marginally worse performance relative to aPRG-adp for all the tested cases, their results do not illustrate in Table \ref{table-Nash}.

\begin{table}[htpb]
\caption{Comparison results of aGRAAL, FRB-L and aPRG-adp on the Nash-Cournot equilibrium problems with different values of $(n, seed)$.}
\label{table-Nash}
\center
\small
\begin{tabular}{|c|c|rccc|rcc|rcc|}
\hline
 &  &\multicolumn{4}{|c|}{FRB-L}&\multicolumn{3}{|c|}{aGRAAL}& \multicolumn{3}{|c|}{aPRG-adp}\\
$m$&\texttt{seed} &\texttt{Iter} &\texttt{Time} &\#LS & $r_n$ &\texttt{Iter}& \texttt{Time}& $r_n$ &\texttt{Iter}& \texttt{Time}& $r_n$ \\
\hline
 \multirow{5}{*}{500}&  1  &345871  &25.1  &345878  &1.0e-8    &409172  &22.3   &1.0e-8   &163927  &8.2 	&1.0e-8\\
 &  2   &7894  &0.6  &7893  &1.0e-8   &9616  &0.5   &1.0e-8    &3946  &0.2  	& 1.0e-8	\\
 &  3   &500000  &35.6  &500015  &\textbf{3.5e-5}    &500000  & 26.7  &\textbf{7.2e-5}     &500000  & 24.1 	&\textbf{1.7e-8}   	\\
 &  4 &6894  &0.5  &6891  &1.0e-8    &8513  &0.5   &1.0e-8      &3920  & 0.2 	&9.8e-9	\\
 &  5 &500000  &35.3  & 500016 & \textbf{2.4e-3}   &500000  &26.3   &\textbf{2.9e-3}      &500000  &23.7  	&\textbf{9.7e-5}   	\\
\hline
\multirow{5}{*}{1000}&  1  &44330  &4.3  &44333  &1.0e-8    &51674  &3.8   &1.0e-8     &20914  &1.4  	&1.0e-8	\\
 &  2   &500000  &47.6  &500014  &\textbf{6.8e-4}   &500000  &36.8   &\textbf{9.0e-4}          &500000  &32.3  	&\textbf{8.3e-6} 	\\
 &  3   &500000  &47.0  &500020  &\textbf{5.5e-2}   &500000  &36.9   &\textbf{6.6e-2}     &500000  &32.1  	&\textbf{4.4e-2}   	\\
 &  4 &220780  &21.2  &220791  &1.0e-8   &261572  &19.3   &1.0e-8     &106766  &7.1  	&1.0e-8  	\\
 &  5 &151578  &14.4  &151584  &1.0e-8    &175447  &12.8   &1.0e-8    &72194  &4.7  	&1.0e-8	\\
\hline
\end{tabular}
\end{table}

We can see from Table \ref{table-Nash} and Figure \ref{Fig:Nash} that aPRG-adp substantially outperforms the other tested methods. Specifically, both aGRAAL and aPRG-adp are free of line searches. The iteration count and CPU time required by aPRG-adp are roughly half of those for aGRAAL. In terms of iterations, FRB-L outperforms aGRAAL. Nevertheless, FRB-L requires on average about one extra linesearch trial per iteration, which leads to higher CPU consumption for FRB-L.

\begin{figure}[htbp]
\centering
\subfigure[$m=500$, \texttt{seed}=1]{
\includegraphics[width=0.45\textwidth]{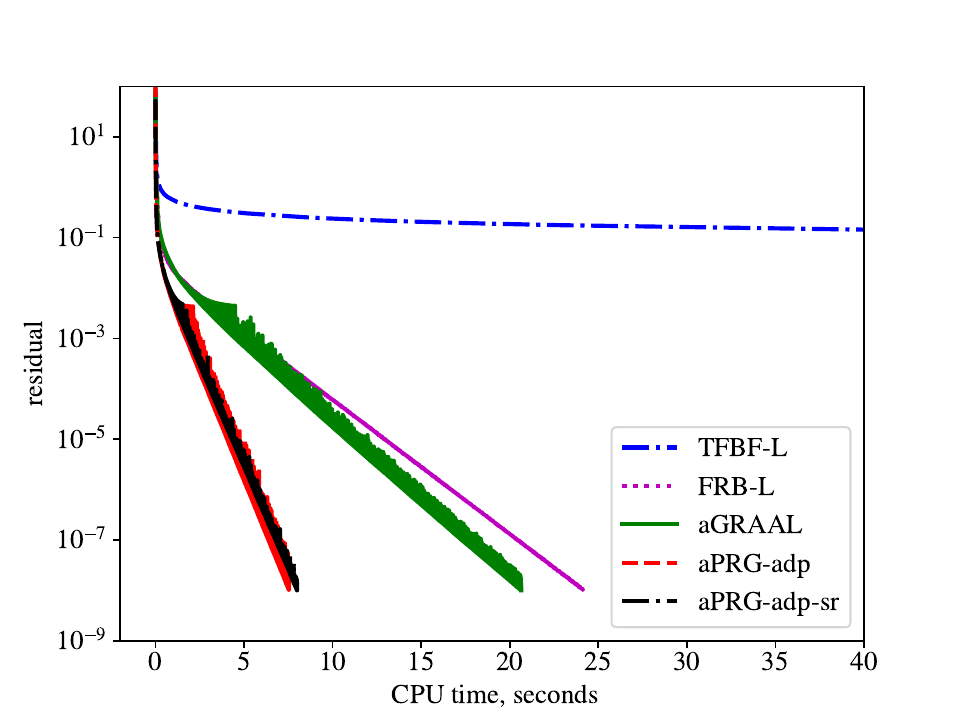}}
\subfigure[$m=1000$, \texttt{seed}=2]{
\includegraphics[width=0.45\textwidth]{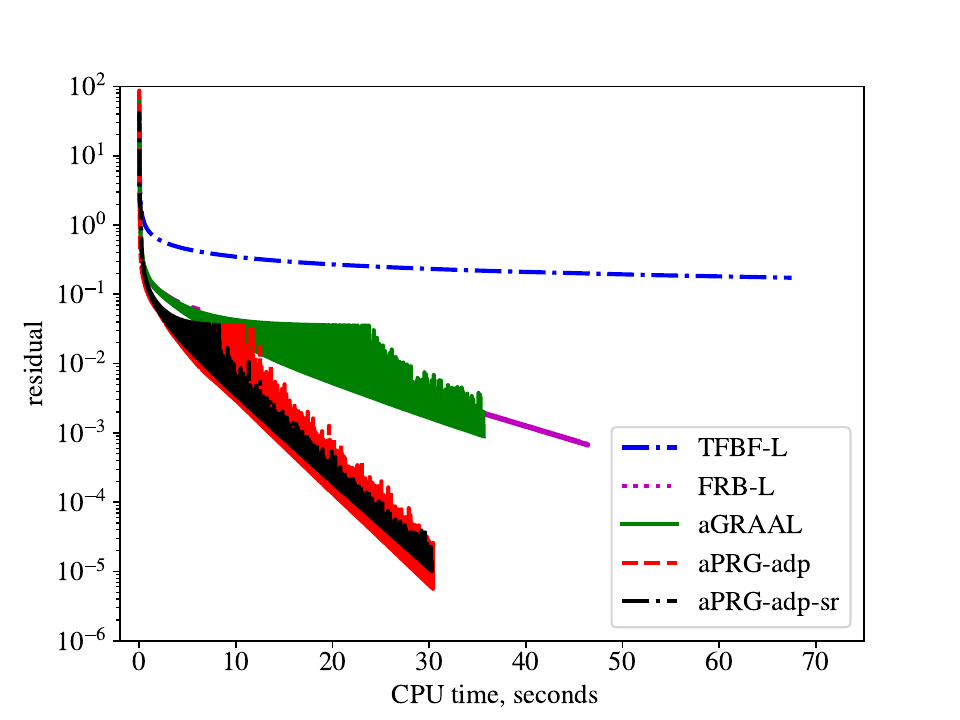}}
\caption{Results of residual for Nash--Cournot equilibrium problem. }
\label{Fig:Nash}
\end{figure}

In Figure \ref{Fig:Nash-step}, we plot the step-size magnitudes for the first 500 iterations obtained from the simulations to analyze the variation of the step-size sequence. Furthermore, we adopt sliding window averaging (denoted by ``-ave") with different window sizes to smooth the step-size sequence and facilitate refined analysis. As observed from Figures \ref{Fig:Nash-step} (a) and \ref{Fig:Nash-step}(c), once the step size falls below a certain threshold, it grows monotonically with variable speed until hitting a higher threshold. Subsequently, the step size enters a descending phase of indeterminate length. Once the step size crosses the lower threshold again, this cycle recurs. This oscillatory phenomenon consistently appears across all numerical simulations. Comparing with aGRAAL, aPRG-adp allows for a wider range and a larger growth rate of step sizes, and yields larger step sizes after certain iterations, as shown in Figures \ref{Fig:Nash-step} (b) and \ref{Fig:Nash-step}(d). Furthermore, we observe that the step sizes of aGRAAL are largely restricted by its inherent growth rate, which constitutes the primary reason for its unsatisfactory numerical performance.

\begin{figure}[htbp]
\centering
\subfigure[$m=500$, \texttt{seed}=1, window = 50]{
\includegraphics[width=0.43\textwidth]{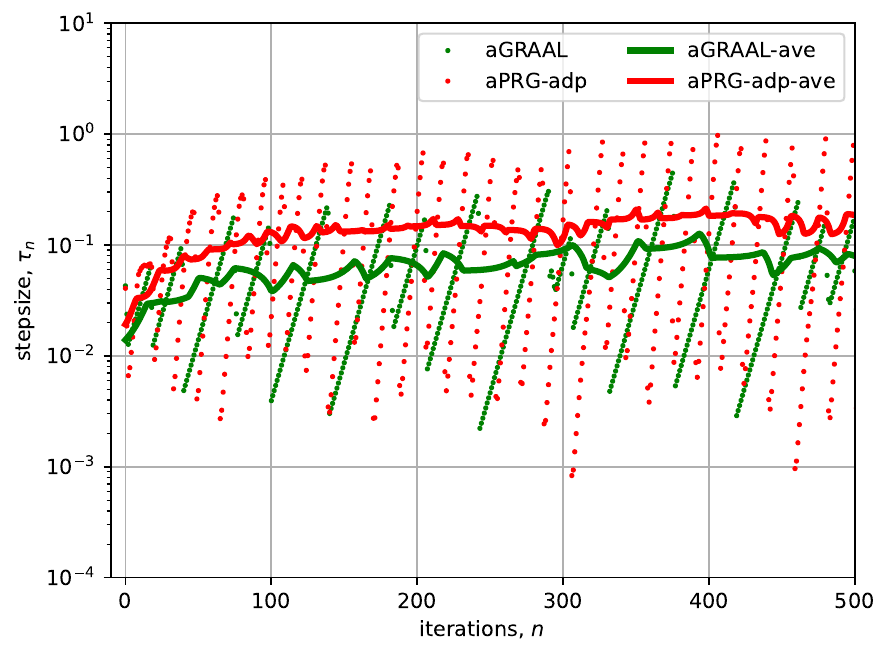}}
\subfigure[$m=500$, \texttt{seed}=1, window = 100]{
\includegraphics[width=0.43\textwidth]{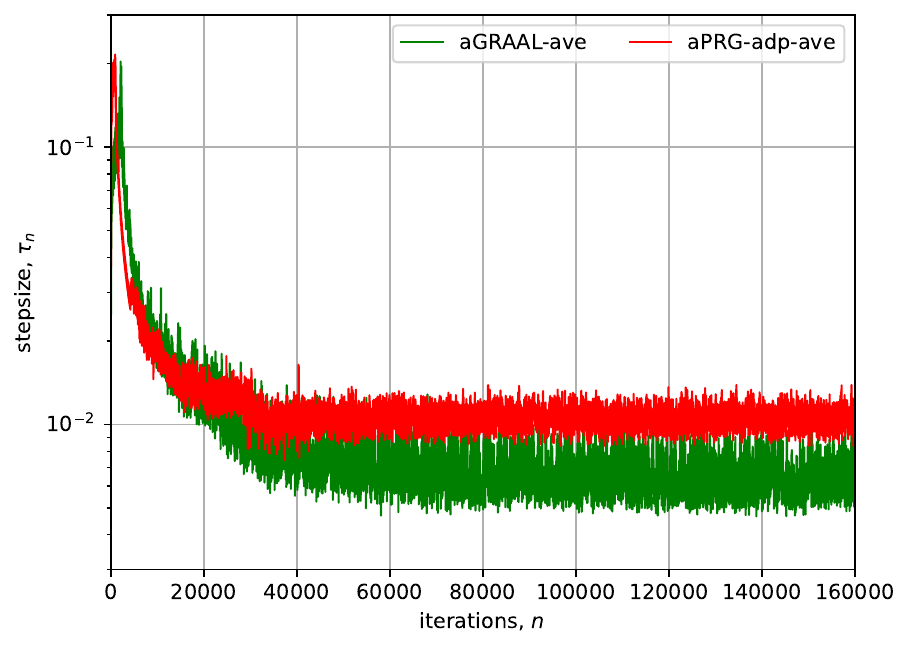}}
\subfigure[$m=1000$, \texttt{seed}=2, window = 50]{
\includegraphics[width=0.43\textwidth]{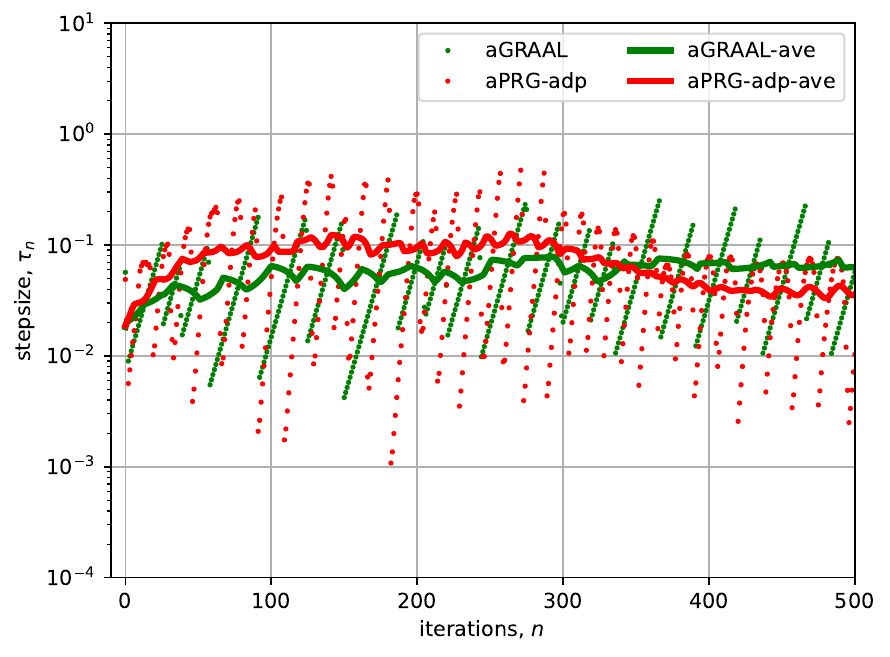}}
\subfigure[$m=1000$, \texttt{seed}=2, window = 300]{
\includegraphics[width=0.43\textwidth]{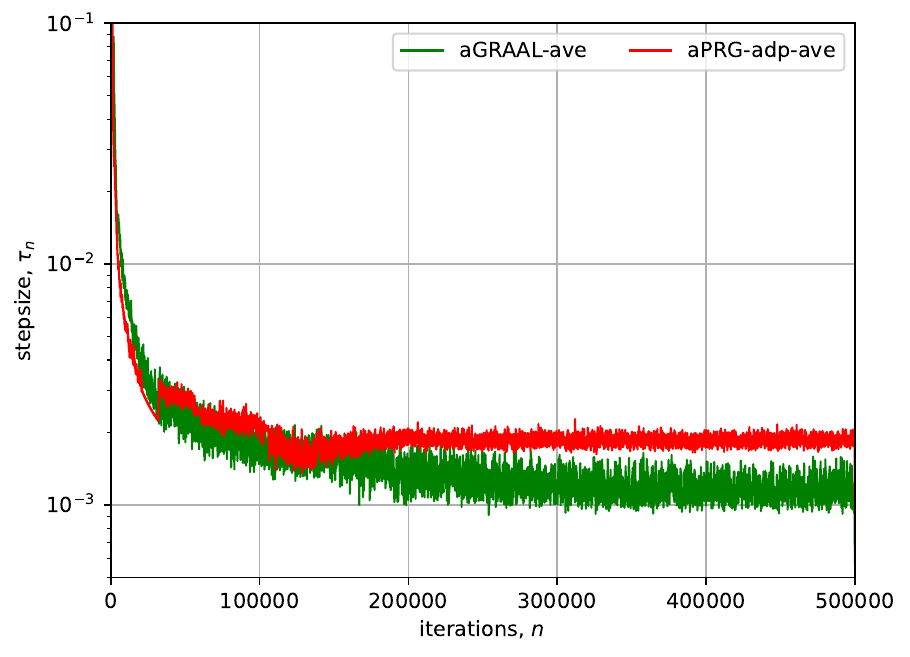}}
\caption{Results for Nash--Cournot equilibrium problem. }
\label{Fig:Nash-step}
\end{figure}

\subsection{HpHard problem}
The HpHard problem is to find $x^\star\in \R^m$, such that
\be\label{HpHard}
\langle x-x^\star, F(x^\star)\rangle \geq 0, \quad \forall x \in C,
\ee
where $F(x) = Mx +q$ with $M = NN^T + S + D$ and $q \in \bR^m$, $N$, $D$ and $S\in\bR^{m\times m}$, $S$ is a skew-symmetric matrix, $D$ is diagonal matrix, whose diagonal entries are nonnegative (so $M$ is positive semidefinite). The set $C$ is called a feasible set.

In this section, every entry of $N$ and $S$ is uniformly generated from $(-5, 5)$, the diagonal entry of $D$ is uniformly generated from $(0, 0.3)$ and every entry of $q$ uniformly generated from $(-500, 0)$. We test two feasible sets:
\ben
C_1 = \{x\in\bR^m_+~|~ \sum_{i=1}^m x_i=m\}~~\mbox{and}~~C_2 = \{x\in\bR^m~|~ x_i\in [-2,5], i=1,2,\dots,m\}.
\een
Since solutions of the problem (\ref{HpHard}) coincide with zeros of the residual $r_n$ defined in \eqref{def:r_n}, thus we test and terminate methods by using the relative residual: $r_n/(1+\|M\|)<\epsilon$ with given $\epsilon>0$ or a maximum
number of iterations was reached ($n_{\max}=1\times 10^5$).

We have generated randomly $M$ and $q$ with $seed=1$ for every fixed $m$. For all tests, we take $x_0=(1,1,\cdots,1)$. Since $F$ is an affine operator, the Lipschitz constant of $F$ can be estimated by $L_F=\|M\|$. Firstly, we tested aPRG ($\tau=0.99/L_F$) and compare it with PRG ($\tau=0.41/L_F$) and GRAAL ($\psi=1.6$ and $\tau=\psi/(2L_F)$). The numerical results are reported in Table \ref{table2} and Figure \ref{Fig:NHpHard-fixed-step}. It can be clearly observed that the proposed averaged PRG outperforms the other competing methods.

\begin{table}[htpb]
\caption{Comparison results of PRG, GRAAL and aPRG with fixed step size and $\epsilon=10^{-6}$ on the HpHard  problems.}
\label{table2}
\center
\small
\begin{tabular}{|c|c|rc|rc|rc|}
\hline
 &  &\multicolumn{2}{|c|}{PRG}& \multicolumn{2}{|c|}{GRAAL}& \multicolumn{2}{|c|}{aPRG}\\
set & $m$ &\texttt{Iter} &\texttt{Time}  &\texttt{Iter}& \texttt{Time} &\texttt{Iter}& \texttt{Time}\\
\hline
 \multirow{3}{*}{$C_1$}&  1000  &21050  &7.9  &28770  &10.9    &17261  &6.4 \\
 &  3000   &38600  &189.0  &52756  &268.0    &31652  &163.1	\\
 &  5000 &54604  &856.0  &74627  &1164.1    &44775  &694.2 	\\
\hline
\multirow{3}{*}{$C_2$}&  200  &18022 &0.6  &24633  &0.8    &14778 &0.5	\\
&  500   &52625  &4.3  &71923  &5.7    &43152  &3.4 	\\
\cline{2-8}
 & \multirow{2}{*}{1000}  &100000  &29.9  &100000  &29.9    &100000  &29.9  	\\
 & & \multicolumn{2}{|c|}{$r_n$=\textbf{4.2e-6}} & \multicolumn{2}{|c|}{$r_n$=\textbf{2.3e-5}}& \multicolumn{2}{|c|}{$r_n$=\textbf{1.1e-6}}\\
\hline
\end{tabular}
\end{table}
\begin{figure}[htbp]
\centering
\subfigure[Set $C_1$ with $m=3000$.]{
\includegraphics[width=0.45\textwidth]{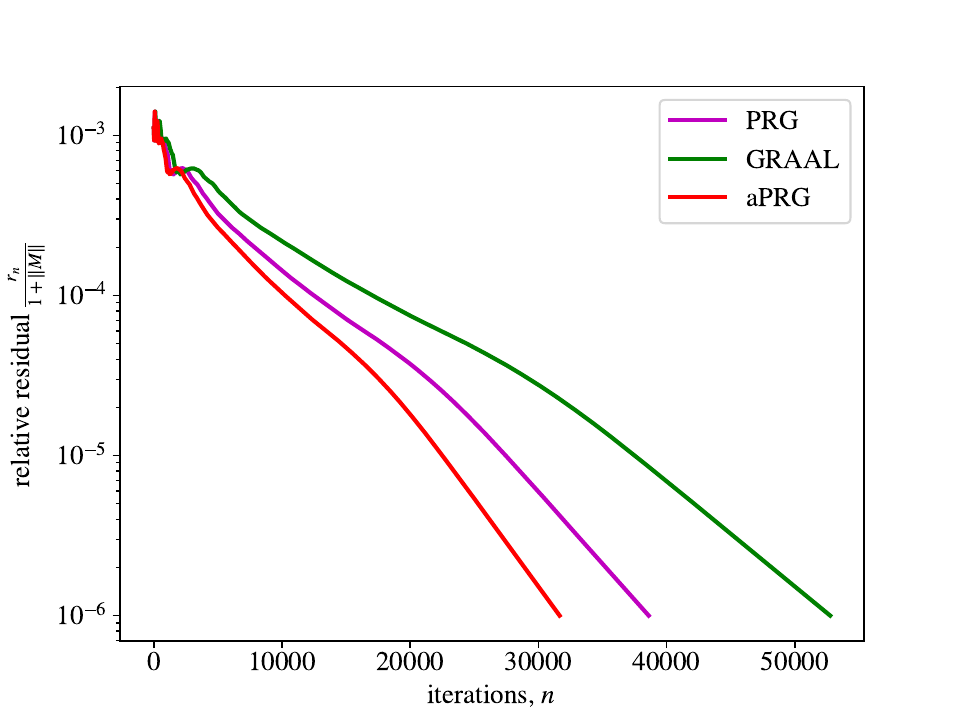}}
\subfigure[Set $C_2$ with $m=500$.]{
\includegraphics[width=0.45\textwidth]{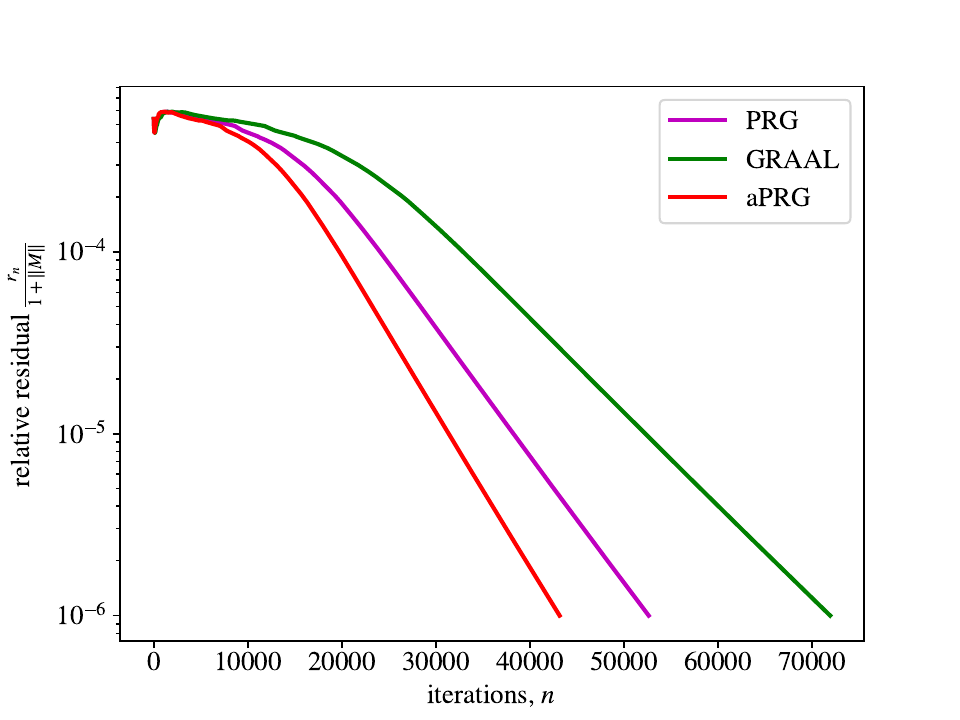}}
\caption{Results of relative residual from the tested methods with fixed step size for HpHard problem.}
\label{Fig:NHpHard-fixed-step}
\end{figure}

Moreover, we tested the adaptive methods for solving the HpHard problems, with the numerical results illustrated in Table \ref{table3} and Figure \ref{Fig:NHpHard-adaptive}. Based on these findings, it can be observed that adaptive methods (including aPRG-adp-sr shown in Figure \ref{Fig:NHpHard-adaptive}) without linesearch exhibit a distinct advantage over those utilizing linesearch in terms of CPU time, despite the fact that TFBF-L and FRB-L require fewer iterations than aGRAAL. Furthermore, the superiority of aPRG-adp and aPRG-adp-sr is particularly pronounced.

\begin{table}[htpb]
\caption{Comparison results of adaptive methods: TFBF-L, FRB-L, aGRAAL and aPRG-adp with $\epsilon=10^{-10}$ on the HpHard  problems.}
\label{table3}
\center
\small
\begin{tabular}{|c|c|rcc|rcc|rc|rc|}
\hline
 &  &\multicolumn{3}{|c|}{TFBF-L}&\multicolumn{3}{|c|}{FRB-L}& \multicolumn{2}{|c|}{aGRAAL}& \multicolumn{2}{|c|}{aPRG-adp}\\
set & $m$ &\texttt{Iter} &\texttt{Time} &\#LS &\texttt{Iter}& \texttt{Time}&\#LS &\texttt{Iter}& \texttt{Time}&\texttt{Iter}& \texttt{Time}\\
\hline
 \multirow{3}{*}{$C_1$}&  1000  &1752  &1.3  &2047  &2131    &1.3  &2135   &2721   &1.0  &1191  	&0.5\\
 &  3000   &1836  &18.9  &2149  &2241    &16.0  &2244   &2871    &13.6  &1258  	&6.0 	\\
 &  5000   &2020  &65.0  &2361  &2492    &58.2  &2496   &3165     &49.5  & 1372 	&21.3  	\\
\hline
\multirow{3}{*}{$C_2$}&  1000  &10151  &6.1  &11595 &15020  &6.7    &15024  &17940   &5.4     &7135  &2.1  	 	\\
 &  3000   &26240  &264.0  &29891  &38874    &283.6  &38878   &45888      &222.2  &18303  	&88.9  	\\
 &  5000   &40634  &1280.4  &46257  &60171    &1373.8  &60175   & 72098    &1099.3  &28602  	&434.6    	\\
\hline
\end{tabular}
\end{table}

\begin{figure}[htbp]
\centering
\subfigure[Set $C_1$ with $m=5000$.]{
\includegraphics[width=0.47\textwidth]{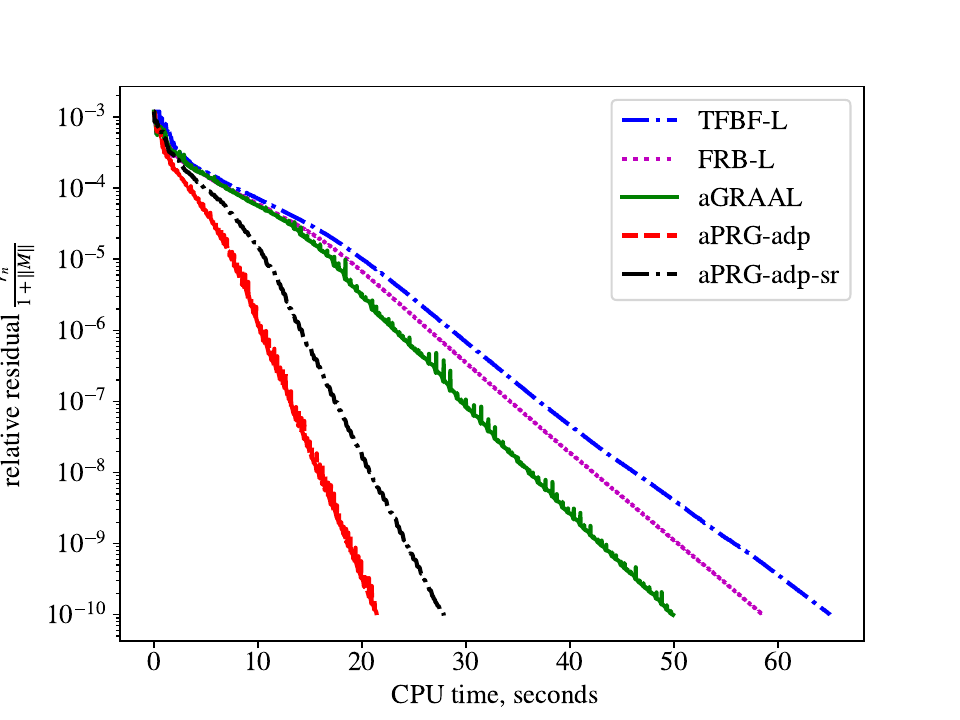}}
\subfigure[Set $C_2$ with $m=5000$.]{
\includegraphics[width=0.47\textwidth]{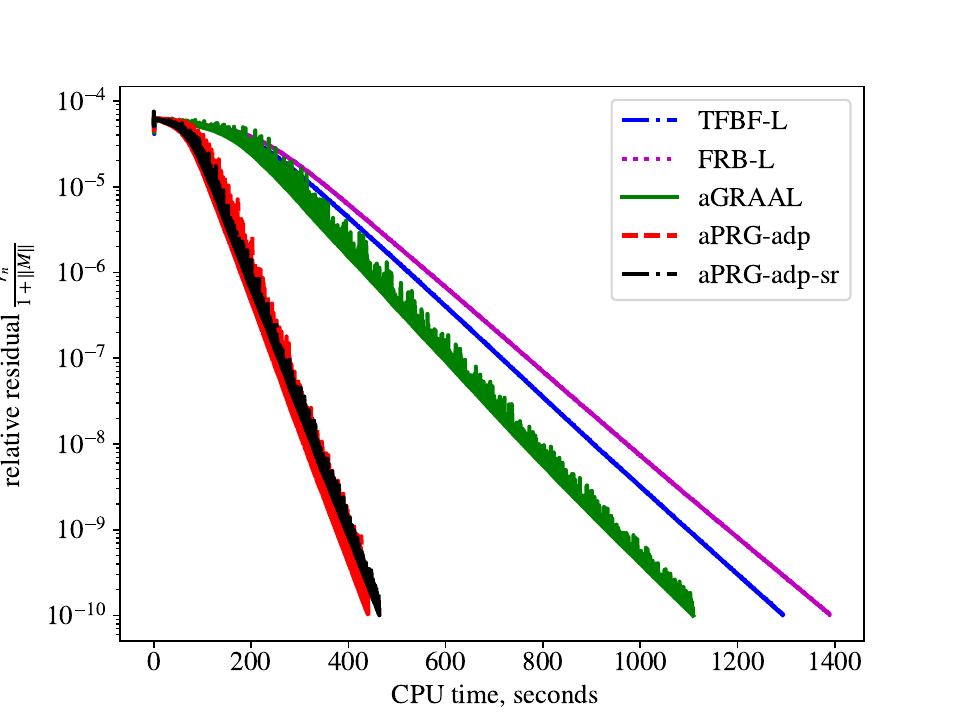}}
\caption{Results of relative residual from the tested methods with adaptive step sizes for HpHard problem.}
\label{Fig:NHpHard-adaptive}
\end{figure}

\subsection{Tomography reconstruction}
The tomography reconstruction problem is an instance of a linear inverse problem
\[
Ax = \hat{b},
\]
where $x \in \mathbb{R}^n$ is the unknown image, $A \in \mathbb{R}^{m \times n}$ is the projection matrix, and $\hat{b} \in \mathbb{R}^m$ is the given sinogram. In practice, however, $\hat{b}$ is contaminated by some noise $\varepsilon \in \mathbb{R}^m$, so we observe only $b = \hat{b} + \varepsilon$.

Let $C_i = \{x: \langle A_i, x \rangle = b_i\}$ for $i=1,\cdots,m$, the tomography reconstruction problem can be formulated as convex feasibility problem: $x\in \bigcap_{i=1}^m C_i$. As the projection onto $C_i$ is given by $P_{C_i}x = x - \frac{\langle A_i, x \rangle - b_i}{\|A_i\|^2} A_i$, computing $Tx = \frac{1}{n} \sum_{i=1}^n P_{C_i}x$ reduces to the matrix-vector multiplications which is realized efficiently in most computer processors. As in \cite{Malitsky2019Golden}, we set $g=0$ and $F(x)=(I-T)(x)$, which is monotone, and solve the corresponding MVI.

We reconstructed the Shepp-Logan phantom image $256 \times 256$ (thus, $x \in \mathbb{R}^n$ with $n = 2^{16}$) from the far less measurements $m = 2^{15}$. We firstly generated the matrix $A \in \mathbb{R}^{m \times n}$ from the \texttt{scikit-learn} library and define $b = Ax + \varepsilon$, where $\varepsilon \in \mathbb{R}^m$ is a random vector, whose entries are drawn from $\mathcal{N}(0,1)$. Then we tested two adaptive methods: aPRG-adp and aGRAAL. The starting point was chosen as $x^1 = (0,\dots,0)$ and $\tau_0 = 1$.

\begin{figure}[htbp]
\centering
\subfigure[Residual]{
\includegraphics[width=0.43\textwidth]{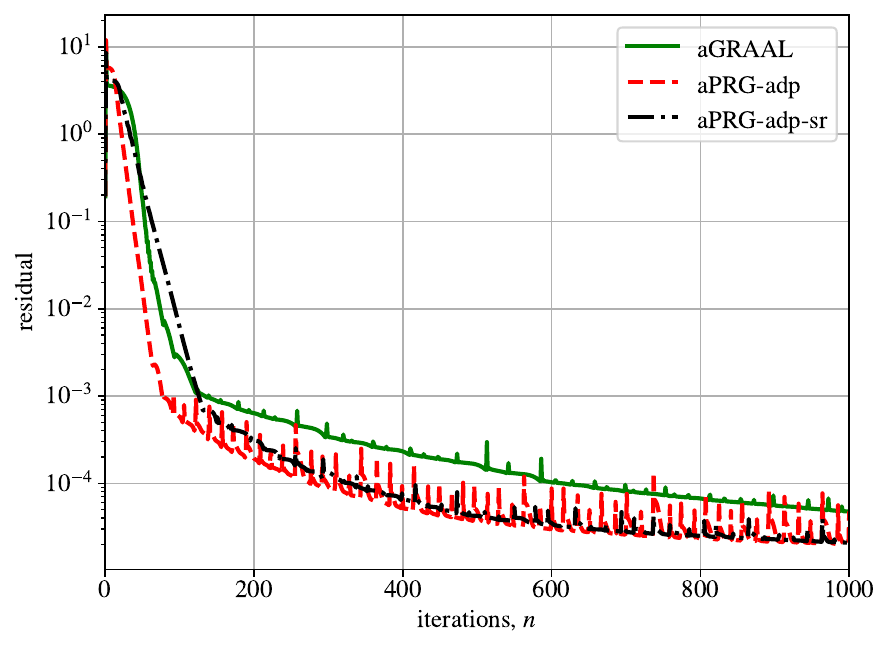}}
\subfigure[Step sizes]{
\includegraphics[width=0.43\textwidth]{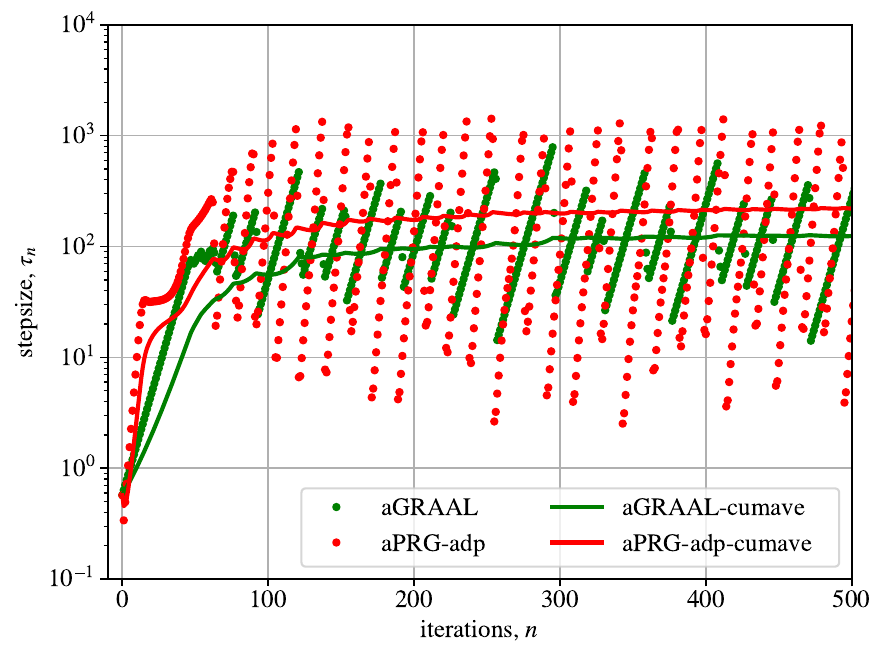}}
\caption{Results for Tomography reconstruction problem.}
\label{Fig:tomo}
\end{figure}

In Figure \ref{Fig:tomo} (a) we plot the residual $\|F(x_n)\|=\|x_n - Tx_n\|$ from aGRAAL and two adaptive aPRGs with respect to the number of iterations, while Figure \ref{Fig:tomo} (b) presents the step-size magnitudes from aGRAAL and aPRG-adp for the first 500 iterations derived from the simulations.
Compared with aGRAAL, aPRG-adp exhibits superior performance: it accommodates a wider range of step sizes and consistently yields larger step sizes, which contributes to its faster convergence. It can be readily observed from Figure \ref{Fig:tomo} (b)  that the step sizes of aGRAAL are largely constrained by its growth rate, which results in its degraded numerical performance.

\section{Conclusions and further directions}
\label{sec_conclusion}
In this paper, we proposed, analyzed, and tested an averaged proximal reflected gradient (aPRG) method for solving monotone variational inequality (MVI) problems. Under a fixed step size derived from the global Lipschitz constant, we theoretically improve the convergence condition, such that the step size satisfies $\tau\in(0, 1/L_F)$. This condition is compatible with that of classical methods, including Popov's extragradient method and the forward-reflected-backward method. Furthermore, we present an adaptive and efficient step size estimation strategy that eliminates the need for line searches. This strategy incorporates a parameter that directly regulates both the step size generation mechanism and the Lyapunov function used for convergence analysis.

\section*{Acknowledgements}
The authors would like to express gratitude towards the authors of \cite{Malitsky2019Golden} for sharing their codes, which were used for fair comparisons in this study.

\bibliographystyle{abbrv}
\bibliography{cxktex}

\end{document}